\documentclass[10.5pt]{elsarticle}
\usepackage[margin=1in]{geometry}

\usepackage{nicematrix}
\usepackage[new]{old-arrows}

\usepackage[T1]{fontenc}
\usepackage{amsmath,amssymb,latexsym,mathrsfs}
\usepackage{mathabx} %for \widecheck
\usepackage{xcolor}
\usepackage[all,color,cmtip]{xy}
\usepackage{rotating}
\usepackage{amsthm}  
\usepackage{mathtools}
\usepackage{enumitem}
\usepackage{titlesec}
\usepackage{tikz}
\biboptions{sort&compress}
\usepackage{blkarray}
\usepackage{hyperref} %this package is included to make the links in the text "clickable"; if you don't want this, just comment this line.

\usepackage{csquotes}

\usepackage{xcolor}

\usepackage{bbm}

\titleformat{\section}
  {\normalfont\fontsize{12}{15}\bfseries}{\thesection}{1em}{}

\titleformat{\subsection}
  {\normalfont\fontsize{11}{15}\bfseries}{\thesubsection.}{0.5em}{}

\def\wh#1{\widehat{#1}}
\def\wt#1{\widetilde{#1}}
\def\TO#1{\overrightarrow{#1}}
\def\OT#1{\overleftarrow{#1}}

\newcommand{\Hom}{\mathrm{Hom}}

\newcommand{\Ext}{\mathrm{Ext}}
\newcommand{\rad}{\mathrm{rad}}
\newcommand{\soc}{\mathrm{soc}}

\newcommand{\End}{\mathrm{End}}
\newcommand{\dimk}{\mathrm{dim}}

\newcommand{\Mimod}{\mathrm{mod}}
\newcommand{\SlMimod}{\mathrm{\text{--}mod}}

\newcommand{\Ker}{\mathrm{Ker} }

\newcommand{\Img}{\mathrm{Im} }
\newcommand{\add}{\mathrm{add}}

\newcommand{\tr}{\mathrm{tr}}
\newcommand{\Tr}{\mathrm{Tr}}

\newcommand{\Hasse}{\mathrm{H}}

\newcommand{\Orb}{\mathrm{orb}}

\newcommand{\sou}{\mathrm{s}}
\newcommand{\tar}{\mathrm{t}}
\newcommand{\poset}{\mathcal{P}}

\newcommand{\Ind}{\mathrm{Ind}}
\newcommand{\Cnd}{\mathrm{CoInd}}
\newcommand{\Id}{\mathrm{Id}}
\newcommand{\Mat}{\mathbb{M}}
\newcommand{\bas}{\mathrm{e}}
\newcommand{\unit}{\mathrm{1}}
\newcommand{\ones}{\mathbbm{1}}

\newcommand{\Quiver}{\mathrm{Q}}
\newcommand{\Paths}{\mathrm{Paths}}
\newcommand{\Card}{\mathrm{Card}}

\newcommand{\res}{\mathrm{res}}

\def\eps{\mathcal{E}}

\newcommand\mathmiddlescript[1]{\vcenter{\hbox{$\scriptstyle #1$}}}

\newtheorem{lemma}{Lemma}[section]{\bf}{\it}
{\bf}{\it}
\newtheorem{corollary}[lemma]{Corollary}{\bf}{\it}
\newtheorem{theorem}[lemma]{Theorem}{\bf}{\it}
{\bf}{\it}
{\bf}{\it}
{\bf}{\it}
{\bf}{\it}
{\bf}{\it}
\theoremstyle{definition}
\newtheorem{remark}[lemma]{Remark}{\bf}{\rm}
{\bf}{\rm}
\newtheorem{definition}[lemma]{Definition}{\bf}{\rm}
{\bf}{\rm}
{\bf}{\rm}

\theoremstyle{remark}

\begin{document}
\title{Burt-Butler algebras of the bocs associated to a finite partially ordered set}
\author{Raymundo Bautista\fnref{CCM}}
\ead{raymundo@matmor.unam.mx}

\author{Jes\'us Arturo Jim\'enez Gonz\'alez\fnref{UNAM,cora}}
\ead{jejim@im.unam.mx}

\fntext[CCM]{Centro de Ciencias Matem\'aticas UNAM Campus Morelia C.P.~58089,  Mexico}
\fntext[UNAM]{Instituto de Matem\'aticas UNAM Ciudad~Universitaria C.P.~04510, Mexico.}

\fntext[cora]{Corresponding author}
%\journal{Springer}

%------------------------------------------------------------------
%------------------------------------------------------------------
%------------------------------------------------------------------
\begin{abstract}
Given an algebra $A$ and an $A-A$-bimodule $U$ with co-algebra structure, a bocs, the algebras of endomorphisms of $A$ as left or right module of the bocs are known as Burt-Butler algebras (up to an appropriate opposite). Here we give a description of these algebras for the bocs associated to a finite partially ordered set in terms of incidence algebras and their balanced versions. We also exhibit their quasi-hereditary structure, provide bound quiver presentations for their Ringel duals, describe the embedding of $A$ as exact Borel subalgebra and characterize the corresponding subcategories of induced and co-induced modules.
\end{abstract}


\maketitle

%\tableofcontents

%--------------------------------------
%--------------------------------------
%--------------------------------------
%--------------------------------------
\section*{Introduction}\label{S(I)}

Representations of partially ordered sets (posets) have been an important part in the development of representation theory since their introduction by Nazarova-Roiter~\cite{NR}, mainly for their role concerning the second Brauer-Thrall conjecture for perfect fields~\cite{NR75,cmR80,zZ05}, and in relation to the representation type of quivers and algebras, see for instance~\cite{pG72,jaD74,GNRSV93}. Matrix representation of posets, as a particular class of matrix problems, can be systematized in the categorical framework of bocses and their representations introduced by Roiter in~\cite{avR80}.  In many relevant cases the category of modules of a bocs is equivalent to certain subcategories of (\emph{induced} or \emph{co-induced}) modules of the right and left Burt-Butler algebras of the bocs, studied in~\cite{BB90}, a relation exploited in the work of Koenig-K\"ulshammer-Ovsienko~\cite{KKO14} on quasi-hereditary algebras and their exact Borel subalgebras. One of the central results in~\cite{KKO14} establishes the right and left Burt-Butler algebras of a \emph{directed} bocs as quasi-hereditary algebras, Ringel dual of each other up to Morita equivalence.

For an arbitrary field $K$, here we compute the right $\mathcal{R}$ and left $\mathcal{L}$ Burt-Butler algebras of the bocs $(A,U)$ associated to a finite poset $\poset$, providing isomorphisms
\[
\mathcal{R} \cong \begin{pmatrix} K\poset & 0 \\ [K, \ldots, K]&K \end{pmatrix}  \longhookleftarrow  \begin{pmatrix} R' & 0 \\ [K, \ldots, K]&K \end{pmatrix}
 \cong A \cong 
 \begin{pmatrix} R' & 0 \\ R'& K\Id \end{pmatrix}  \longhookrightarrow  \begin{pmatrix} K\poset & K^{\ones}_0\poset \\ K\poset&K^{\ones}\poset \end{pmatrix} \cong \mathcal{L},
\]
where $K\poset$ is the incidence algebra of $\poset$, $R'$ is its diagonal subalgebra and $K^{\ones}\poset:=\{M \in K\poset \mid M\ones \in K\ones\}$ with $\ones:=[1,\ldots,1]^{\tr}$ (resp. $K^{\ones}\poset_0:=\{M \in K\poset \mid M\ones=0\}$) is its row-balanced subalgebra (resp. row-centralized subspace), see Theorems~\ref{T:RBB} and \ref{T:LBB}. Our computations rely on faithful actions of $\mathcal{R}$ and $\mathcal{L}$ on the base algebra $A$, generated by appropriate grouplike elements of the bocs. Moreover, since the bocs $(A,U)$ is directed (Lemma~\ref{P:directed}), the algebras $\mathcal{R}$ and $\mathcal{L}$ have quasi-hereditary structures (Table~\ref{Tb:QHstructure}) that we further analyze as follows:

\begin{itemize}
 \item In Remark~\ref{R:basic} we describe the Ringel dual of $\mathcal{R}$, that is, a basic algebra $\mathcal{L}'$ Morita equivalent to $\mathcal{L}$ (see Corollary~\ref{C:Ringel} and Definition~\ref{D:LBQ} for a bound quiver presentation of $\mathcal{L}'$).
 \item In Corollary~\ref{C:Borel} we describe the base algebra $A$ as exact and co-exact Borel subalgebra of $\mathcal{R}$ and $\mathcal{L}$ respectively, as indicated in the isomorphisms above.
 \item Since the incidence algebra $K\poset$ is contained in both the right and left algebras $\mathcal{R}$, $\mathcal{L}$ (see isomorphisms above), in Theorem~\ref{T:COinduced} we characterize the induced and co-induced modules from $A$ as those whose restriction to $K\poset\SlMimod$ is a projective or injective module, respectively.
\end{itemize}

Note that the right algebra $\mathcal{R}$ is isomorphic to the incidence algebra $K\poset^*$ of the (upper) suspension of $\poset$, which has been extensively studied in relation to representations of posets, see~\cite{dS92book} and references therein. For instance, induced modules are called \emph{prinjective} in~\cite[11.6]{dS92book}, and Theorem~\ref{T:COinduced}(a) establishes directly the well-known equivalence between the full subcategory of $K\poset^*\SlMimod$ determined by prinjective modules and the additive category of matrix representations of $\poset$ (see~\cite[Proposition~11.36]{dS92book}). Our presentation of co-induced modules as model for such category, Theorem~\ref{T:COinduced}(b), seems to be new. Section~\ref{N} collects preliminary notions and results, in particular related to row-balanced incidence algebras, and section~\ref{S(X)} contains bound quiver presentations and examples.

%--------------------------------------
%--------------------------------------
%--------------------------------------
%--------------------------------------
\section{Preliminary notions} \label{N}

For integers $m,n \geq 1$, the set of $m \times n$ matrices with coefficients in an arbitrary field $K$ is denoted by $\Mat_{m \times n}(K)$, and simply by $\Mat_n(K)$ if $m=n$. We consider the usual $K$-basis $\{E_{ji}\}$ of $\Mat_{m \times n}(K)$, and denote by $\{M_{ji}\}$ the corresponding coefficients of a matrix $M$. The linear span of a set of vectors $X$ in a $K$-vector space is denoted by $\langle X \rangle$. Quivers are denoted by $(Q_0,Q_1,\sou,\tar)$ where $\sou,\tar:Q_1 \to Q_0$ are the source and target functions, and arrows are composed from right to left.

%--------------------------------------
%--------------------------------------
\subsection{Bocses and their Burt-Butler algebras} \label{N1}

A \emph{bocs} (or \emph{bimodule over a category with co-algebra structure}) $\mathcal{B}=(A,U)$ consists of a $K$-algebra $A$ and an $A-A$-bimodule $U$ together with morphisms of $A-A$-bimodules $\mu:U \to U \otimes_A U$ and $\epsilon:U \to A$ such that the following diagrams are commutative,
\[
\xymatrix{U \otimes_A U \ar[d]_-{\unit \otimes \epsilon} & U \ar@{=}[d] \ar[r]^-{\mu} \ar[l]_-{\mu} & U \otimes_A U \ar[d]^-{\epsilon \otimes \unit} \\
U \otimes_A A \ar[r]^-{\cong} & U  & A \otimes_A U \ar[l]_-{\cong}} \qquad \text{and} \qquad
\xymatrix{ U \ar@{=}[d] \ar[r]^-{\mu} & U \otimes_A U \ar[r]^-{\unit \otimes \mu} & U \otimes_A (U \otimes_A U) \ar[d]^-{\cong} \\
 U \ar[r]^-{\mu} & U \otimes_A U \ar[r]^-{\mu \otimes \unit} & (U \otimes_A U) \otimes_A U. }
\]
The morphisms $\mu$ and $\epsilon$ are called \emph{comultiplication} and \emph{counit} of $\mathcal{B}$ respectively. The category of left modules $\mathcal{B}\SlMimod$ of the bocs $\mathcal{B}$ has as objects the finite dimensional left $A$-modules. For two modules $M,N$ the set of morphisms $\Hom_{\mathcal{B}}(M,N)$ is $\Hom_A({}_AU\otimes_A M,{}_AN)$ and the composition of $f:U\otimes_A M \to N$ and $g:U\otimes N \to L$ is given by the composition of morphisms
\begin{equation*} \label{Eq:compositionL}
 \xymatrix{U \otimes_A M \ar[r]^-{\mu \otimes \unit} & U \otimes_A U \otimes_A M \ar[r]^-{\unit \otimes f} 
& U \otimes_A N \ar[r]^-{g} & L }.
\end{equation*}
The identity of $M$ in $\mathcal{B}\SlMimod$ is given by the composition of $\epsilon \otimes \unit_M$ with the canonical isomorphism $A \otimes_A M \cong M$. Similarly, the objects of the category $\Mimod \text{-}\mathcal{B}$ of right $\mathcal{B}$-modules are the finite dimensional right $A$-modules, morphisms are given by $\Hom_{\mathcal{B}}(M,N):=\Hom_A(M\otimes_AU_A,N_A)$ and the composition of $\varphi:M \otimes_AU \to N$ and $\psi:N \otimes_A U \to L$ is given by
\begin{equation*} \label{Eq:compositionR}
 \xymatrix{M \otimes_A U \ar[r]^-{\unit \otimes \mu} & M \otimes_A U \otimes_A U \ar[r]^-{\varphi \otimes \unit} & N \otimes_A U \ar[r]^-{\psi} & L }.
\end{equation*}
The identity of $M_A$ in $\Mimod \text{-}\mathcal{B}$ is the composition of $\unit_M \otimes \epsilon$ with the canonical isomorphism $M \otimes_A A \cong M$.

The \emph{right and left Burt-Butler algebras} associated to a bocs $\mathcal{B}=(A,U)$ are given by 
\[
\mathcal{R}:=\End_{\mathcal{B}}({}_AA)^{op} \cong \Hom_A({}_AU,{}_AA)^{op} \qquad \text{and} \qquad \mathcal{L}:=\End_{\mathcal{B}}(A_A) \cong \Hom_A(U_A,A_A),
\]
cf.~\cite{KKO14}. Using the formulas of the composition, the product for elements $f,g$ in $\mathcal{R}$ and $\varphi,\psi$ in $\mathcal{L}$ is given by
\begin{eqnarray} \label{Eq:1}
(fg)(u)=\sum_ig(u_i^1f(u_i^2))\qquad\text{and}\qquad  (\psi \varphi )(u)=\sum_i\psi(\varphi(u_i^1)u_i^2), 
\end{eqnarray}
for an element $u \in U$ with $\mu(u)=\sum_iu_i^1 \otimes u_i^2$. Moreover, the counit $\epsilon$ is the identity element of both algebras $\mathcal{R}$ and $\mathcal{L}$. Applying the contravariant functor $\Hom_A(-,{}_AA)$ to the counit $\epsilon:{}_AU \to {}_AA$ we get a linear transformation
\[
A^{op} \cong \xymatrix@R=4pt{\Hom_A({}_AA,{}_AA) \ar[rr]^-{\Hom(\epsilon,{}_AA)} && \Hom_A({}_AU,{}_AA) \\ f_z:=[a \mapsto az] \ar@{|->}[rr] && f_z \circ \epsilon = [u \mapsto \epsilon(u)z] } \cong \mathcal{R}^{op},
\]
and under the given isomorphisms we write $\TO{z}:=(f_z \circ \epsilon) \in \mathcal{R}$. Observe that this transformation yields a morphism of $K$-algebras $A \to \mathcal{R}$, since $\TO{\unit}=\epsilon$ and
\[
(\TO{z_1}\, \TO{z_2})(u)=\sum_i\TO{z_2}(u^1_i\TO{z_1}(u^2_i))=\sum_i \epsilon(u^1_i\epsilon(e^2_i))z_1z_2=\epsilon(u)z_1z_2=\TO{z_1z_2}(u).
\] 
Similarly, the contravariant functor $\Hom_A(-,A_A)$ applied to the counit $\epsilon:U_A \to A_A$ yields a morphism of $K$-algebras
\[
A \cong \xymatrix@R=4pt{\Hom_A(A_A,A_A) \ar[rr]^-{\Hom(\epsilon,A_A)} && \Hom_A(U_A,A_A) \\ g_z:=[a \mapsto za] \ar@{|->}[rr] && g_z \circ \epsilon = [u \mapsto z\epsilon(u)] } \cong \mathcal{L},
\]
that we denote by $z \mapsto \OT{z}:=g_z \circ \epsilon$. Note that if the counit $\epsilon$ is surjective then $\Hom(\epsilon,{}_AA)$ and $\Hom(\epsilon,A_A)$ are monomorphisms. In this case we denote by $\TO{A}$ and $\OT{A}$ the subalgebras of $\mathcal{R}$ and $\mathcal{L}$ isomorphic to $A$ under these monomorphisms, respectively.

%--------------------------------------
%--------------------------------------
\subsection{Grouplike elements and ditalgebras} \label{Nmedio}

A bocs $(A,U)$ is called \emph{normal} if there is an element $\omega \in U$ satisfying $\mu(\omega)=\omega \otimes \omega$ and $\epsilon(\omega)=1$, called a \emph{grouplike} element of $(A,U)$, cf.~\cite[Definition~3.1]{BSZ} and~\cite[\S 7]{KKO14}. In this case the counit $\epsilon$ is necessarily surjective, and $\TO{z}(\omega)=z=\OT{z}(\omega)$ for any $z \in A$. Each grouplike element $\omega$ gives rise to a right action of $\mathcal{R}$ on $A$ by taking
\[
z \cdot f:=(\TO{z}f)(\omega)=f(\omega z), \qquad \text{for $z \in A$ and $f \in \mathcal{R}$.}
\]
Indeed, since $\mu(\omega z)=\omega \otimes \omega z$ we have $(z \cdot f_1)\cdot f_2=f_2(\omega f_1(\omega z))=(f_1f_2)(\omega z)=z \cdot (f_1f_2)$ using~(\ref{Eq:1}). Moreover, this action extends the natural right action of $A$ on $A_A$, in the sense that $z_0 \cdot \TO{z}=(\TO{z_0}\,\TO{z})(\omega)=\TO{z_0z}(\omega)=z_0z$. Similarly, each grouplike element $\omega$ gives rise to a left action of $\mathcal{L}$ on $A$ that extends the natural action of $A$ on ${}_AA$, by taking
\[
g \cdot z:=(g\OT{z})(\omega)=g(z\omega ), \qquad \text{for $z \in A$ and $g \in \mathcal{L}$.}
\]
We call these the right and left \emph{$\omega$-action} of $\mathcal{R}$ and $\mathcal{L}$ on $A$, respectively.

Normal bocses can be conveniently constructed and studied from a ring theoretical setting through tensor algebras and differentials, cf.~\cite[\S 3]{BSZ}. Recall that a \emph{ditalgebra} is a pair $(T,\delta)$ consisting of:
\begin{enumerate}[label={\textnormal{(\alph*)}},topsep=3px,parsep=0px]
 \item A positively graded $K$-algebra $T=\bigoplus_{i \geq 0}T_i$ that is freely generated by the pair $(T_0,T_1)$. In other words, $T$ is isomorphic to the tensor algebra $T_{T_0}(T_1)$ and with this isomorphism the grading of $T$ coincides with the tensor product grading of $T_{T_0}(T_1)$.  
 \item A differential $\delta$ on $T$ with $\delta^2=0$ (here a differential means a homogeneous linear transformation $\delta:T \to T$ of degree one satisfying the Leibniz rule,
\[
\delta(ab)=\delta(a)b+(-1)^{|a|}a\delta(b), \qquad \text{for any homogeneous vectors $a,b \in T$, where $|a|:=i$ if $a \in T_i$).}
\]  
\end{enumerate}
The bocs associated to a ditalgebra $(T,\delta)$ is given as follows: take $A:=T_0$, $V:=T_1$, and consider the vector space $U:=\omega A \oplus V$ for an external symbol $\omega$, with $A-A$-bimodule structure given for $a \in A$ and $\omega b+v \in U$ by the formulas
\begin{equation} %\label{Eq:actionsL}
 a(\omega b+v)=\omega ab +\delta(a)b +av,  \qquad \text{and} \qquad 
(\omega b+v)a=\omega ba +va. 
\end{equation}
The counit $\epsilon:U \to A$ and comultiplication $\mu:U \to U \otimes U$ are given by
\begin{equation} %\label{Eq:coUnit}
\epsilon(\omega b+v)=b, \qquad \text{and} \qquad \mu(\omega b+v)=\omega \otimes \omega b+\delta(v)+\omega \otimes v+v\otimes \omega.  
\end{equation}
Then $(A,U)$ is a normal bocs with grouplike element $\omega$, and the counit $\epsilon$ is surjective with kernel $V$, see~\cite[Lemma~3.3]{BSZ}.  That every normal bocs can be found in this way is well-known, see for instance~\cite[\S 3]{BSZ}.

%--------------------------------------
%--------------------------------------
\subsection{Incidence algebras} \label{N3}

Let ${\poset}=({\poset},\preceq)$ be a finite partially ordered set (\emph{poset}). Here we will always assume that $\poset=\{1,\ldots,n\}$ for some $n \geq 1$, and that $i \preceq j$ implies $i\leq j$. We say that $j$ is an immediate successor of $i$ if $i \prec j$ and there is no $k \in \poset$ with $i \prec k \prec j$, denoted by $i \to j$  The \emph{incidence algebra} $K\poset$ of $\poset$ is given by
\begin{equation}\label{Eq:incidence}
K\poset:=\{ M \in \Mat_n(K) \mid \text{$M_{ji} \neq 0$ implies $i\preceq j$ in $\poset$} \}, \qquad
K\poset=R' \oplus \rad(K\poset),
\end{equation}
where $R'$ is the set of $n \times n$ diagonal matrices, cf.~\cite[(2.7)]{dS92book}. A basis of $K\poset$ is given by the set $\{E_{ji} \mid \text{$i\preceq j$}\}$. 

Denote by $\ones_n$ the vector $[1,1,\ldots,1,1]^{\tr} \in K^n$ (or simply $\ones$ for appropriate size) and consider the following $K$-vector subspaces of $K\poset$,
\begin{equation}\label{D:IncOnes}
K^{\ones}\poset:=\{ M \in K\poset \mid M\ones \in K\ones\} \qquad \text{and} \qquad
K^{\ones}_0\poset:=\{ M \in K\poset \mid M\ones=0\}.
\end{equation}
Clearly, $K^{\ones}\poset$ is a $K$-subalgebra of $K\poset$ and $K^{\ones}_0\poset$ is a $K\poset-K^{\ones}\poset$-bimodule, that we call \emph{row-balanced incidence algebra} and \emph{row-centralized subspace}, respectively. A $K$-basis of $K^{\ones}_0\poset$ is given by the set $\{E_{jj}-E_{ji} \mid \text{$i \prec j$ in $\poset$}\}$, and clearly we have $K^{\ones}\poset=K\Id_n \oplus K^{\ones}_0\poset$. To describe the radical $\rad(K^{\ones}\poset)$ we need to fix a function $m: \poset \to \min(\poset)$ such that $m(j) \in \min(\poset,j):=\{i \in \min(\poset) \mid i \preceq j\}$, called a \emph{minimal marking} of $\poset$. Take $\poset':=\poset-\min(\poset)$ and for $i\prec j$ consider the matrices
\begin{equation} \label{No:Idem}
\eps_j:=E_{jj}-E_{jm(j)},  \qquad  \eps_0:=\Id-\sum_{j \in \poset'}\eps_j \qquad \text{and} \qquad  B_{ji}:=E_{ji}-E_{jm(j)}.
\end{equation}
It is easy to see that the set $\{\eps_0,\eps_j \mid j \in \poset'\}$ is a \emph{complete set of orthogonal primitive idempotents} of $K^{\ones}\poset$. Note also that $B_{ji}=0$ if and only if $i=m(j)$, and that $B_{ji} \in \rad(K^{\ones}\poset)$ for any $i\prec j$ (since these are nilpotent matrices). This shows that
\begin{equation}\label{Eq:2}
K^{\ones}\poset=R^{\ones} \oplus \rad(K^{\ones}\poset), \qquad \text{where} \quad R^{\ones}=\bigoplus_{j \in \poset' \cup \{0\}}K\eps_j \quad \text{ and } \quad \rad(K^{\ones}\poset)=\bigoplus_{\substack{i\prec j \\ i \neq m(j)}}KB_{ji}.
\end{equation}

%--------------------------------------
%--------------------------------------
%--------------------------------------
%--------------------------------------
\section{Computation of the Burt-Butler algebras associated to a finite poset} \label{B}

The plan of this section is to find appropriate faithful actions of the Burt-Butler algebras on the base algebra $A$. With help of a given $K$-basis of $A$ we give general properties of the resulting embeddings, and counting dimension we determine such algebras up to isomorphism.

%--------------------------------------
%--------------------------------------
\subsection{The bocs of a finite poset} \label{B0}

We follow the construction of subsection~\ref{Nmedio} using a ditalgebra $(T,\delta)$ (see~\cite[34.1,~34.2]{BSZ} and~\cite[2.5]{dS92book}): The $K$-algebra $T$ is the path algebra $KQ$ of the quiver $Q$ with vertices the set $\poset \cup \{0\}$, one arrow $\alpha_i:i \to 0$ for each $i \in \poset$, and one arrow $v_{ij}:j \to i$ for each $i \prec j$. We denote by $e_0,e_1,\ldots,e_n$ the trivial paths of $Q$ (so $\alpha_i=e_0\alpha_ie_i$ for all $i \in \poset$). The grading of $T$ is determined additively by taking $|e_0|=|e_i|=|\alpha_i|=0$ for $i \in \poset$, and $|v_{ij}|=1$ for $i \prec j$. Take
\begin{equation} \label{Eq:layer}
R:=\bigoplus_{i \in \poset \cup \{0\}} K e_i, \qquad W_0:=\bigoplus_{i \in \poset} K \alpha_i \qquad \text{and} \qquad W_1:=\bigoplus_{i \prec j} K v_{ij}.
\end{equation}
It is well known that $KQ \cong T_R(W_0 \oplus W_1)$, and one can similarly verify that $KQ \cong T_A(V)$ where $A:=T_0$ and $V:=T_1$. In this situation, any pair of morphisms of $R-R$-bimodules $\delta_0:W_0 \to KQ$ and $\delta_1:W_1 \to KQ$ with $\delta_0(W_0) \subseteq [KQ]_1$ and $\delta_1(W_1) \subseteq [KQ]_2$ extend uniquely to a differential $\delta:KQ \to KQ$ satisfying $\delta(R)=0$ (cf.~\cite[4.1,4.3,4.4]{BSZ}). In our construction, the differential $\delta$ is then determined by the morphisms of $R-R$-bimodules $\delta_0:W_0 \to T_1$ and $\delta_1:W_1 \to T_2$ given by
 \begin{equation} \label{Eq:differential}
\delta_0(\alpha_i):=-\sum_{h\prec i}\alpha_h v_{hi}, \qquad 
\text{and} \qquad \delta_1(v_{ij})=\sum_{i\prec k\prec j}v_{i k} v_{k j}.
\end{equation}
That $\delta^2=0$ is shown in~\cite[34.1]{BSZ}. Note that $A=R \oplus W_0$ and that for any $i\prec j$ in $\poset$ the two dimensional $K$-vector space $[Kv_{ij} \oplus K\alpha_i v_{ij}]$ is an $A-A$-bimodule isomorphic to $Ae_i \otimes_K e_jA$ with the linear transformation given by $v_{ij} \mapsto e_i\otimes e_j$ and $\alpha_iv_{ij} \mapsto \alpha_i \otimes e_j$. In particular, there is an $A-A$-bimodule decomposition 
\begin{equation} \label{Eq:AV}
V=W_1 \oplus AW_1= \bigoplus_{i \prec j} [Kv_{ij} \oplus K\alpha_i v_{ij}] \cong \bigoplus_{i \prec j} Ae_i\otimes_K e_jA.
\end{equation}

Recall that the vector spaces $\Hom_A({}_AM_R,{}_AN_R)$ and $\Hom_A({}_RM_A,{}_RN_A)$ are $R-R$-bimodules with respective actions given by
\[
r_1fr_2(m):=f(mr_1)r_2 \qquad \text{and} \qquad r_1gr_2(m):=r_1g(r_2m).
\]
Since $\mathcal{R}\cong \Hom_A({}_AU,{}_AA)^{op}$ and $\mathcal{L}\cong \Hom_A(U_A,A_A)$, then this $R$ action coincides with the action of the subalgebra $\TO{R}$ on $\mathcal{R}$, and similarly for the subalgebra $\OT{R}$ of $\mathcal{L}$. We identify the set $\{e_i \mid i \in \poset \cup \{0\}\}$ with complete sets of orthogonal idempotents of $\mathcal{R}$ and $\mathcal{L}$ (not necessarily primitive, as we will see below for the left case). In particular, taking $e':=e_1+\ldots+e_n$ we will consider isomorphisms
\[
\mathcal{R}\cong \begin{pmatrix} e'\mathcal{R}e'&e'\mathcal{R}e_0\\e_0\mathcal{R}e'&e_0\mathcal{R}e_0 \end{pmatrix} \qquad \text{and} \qquad \mathcal{L}\cong \begin{pmatrix} e'\mathcal{L}e'&e'\mathcal{L}e_0\\e_0\mathcal{L}e'&e_0\mathcal{L}e_0 \end{pmatrix}.
\]

%--------------------------------------
%--------------------------------------
\subsection{The right algebra} \label{B1}

Let $M$ be a right $\mathcal{R}$-module with a (finite) $K$-basis $\mathbb{B}$. Then the action of $f \in \mathcal{R}$ on $M$ is determined by the structural constants $\{f^{b,b'}\}_{b,b' \in \mathbb{B}}$ given by the equations $b \cdot f =\sum_{b' \in \mathbb{B}}b'f^{b,b'}$ for $b \in \mathbb{B}$. Assuming that each $b \in \mathbb{B}$ belongs to $Me_i$ for some $i \in \{0,1,\ldots,n\}$, and that $f=e_jfe_i$ for $i,j \in \{0,1,\ldots,n\}$, then $f^{b,b'}=f^{be_j,b'e_i}$ for $b, b' \in \mathbb{B}$, where for convenience we take $f^{b,0}=f^{0,b'}=f^{0,0}=0$. Recall that $\omega$ is the grouplike element of the bocs $(A,U)$ given by its construction via ditalgebras.

\begin{lemma}\label{L:Rfaith}
The right $\omega$-action of $\mathcal{R}$ on $e_0A$ is faithful. Moreover, the structural constants of this action with respect to the $K$-basis $\mathbb{B}_0=\{\alpha_1,\ldots,\alpha_n,e_0\}$ of $e_0A$ satisfy that $f^{\alpha_j,b} \neq 0$ implies $b=\alpha_i$ for some $i \preceq j$, for any $f \in \mathcal{R}$ and any $j \in \poset$.
\end{lemma}
\begin{proof}
Using~(\ref{Eq:differential}), the right $\omega$-action of $f$ on $\alpha_j$ is given by
\begin{equation}\label{Eq:nuevo}
\alpha_j \cdot f = f(\omega \alpha_j)=f(\alpha_j \omega -\delta(\alpha_j))=\alpha_jf(\omega)+\sum_{i\prec j}\alpha_if(v_{ij})=\sum_{i\preceq j}\alpha_if^{\alpha_j,\alpha_i}.
\end{equation}
This shows our last claim on structural constants. We also have $e_0 \cdot f =f(\omega e_0)=e_0f(\omega)=\sum_{b \in \mathbb{B}_0}bf^{e_0,b}$. Assume now that $z \cdot f=0$ for all $z \in e_0A$ for some element $f:{}_AU=A\omega \oplus V \to {}_AA$ of $\mathcal{R}$. Note that $\alpha_jf(\omega)=\alpha_jf(e_j\omega)$ with $f(e_j\omega) \in e_jA=Ke_j$ and $f(v_{ij}) \in e_iA=Ke_i$. Therefore, using~(\ref{Eq:nuevo}) we get $0=f^{\alpha_j,\alpha_j}e_j=f(e_j \omega)$ and $0=f^{\alpha_j,\alpha_i}e_i=f(v_{ij})$. Then $f(V)=0$, and since $0=e_0 \cdot f=f(\omega e_0)=f(e_0\omega)$ we have $f(\omega)=0$, which implies $f=0$. That is, the given action is faithful.
\end{proof}

For a finite poset $\poset$ we denote by $\Card(i\prec j,\poset)$ be the number of pairs $(i,j)$ with $i \prec j$ in $\poset$. Recall that $K\poset$ denotes the incidence algebra of $\poset$, see subsection~\ref{N3}.

\begin{theorem}\label{T:RBB}
For any finite poset $\poset$ there are isomorphisms of $K$-algebras
\[
A \cong\begin{pmatrix} R' & 0 \\ \Mat_{1 \times n}(K)&K \end{pmatrix}  \longhookrightarrow \begin{pmatrix} K\poset & 0 \\ \Mat_{1 \times n}(K)&K \end{pmatrix} \cong \mathcal{R}^{\poset},
\]
where $\mathcal{R}^{\poset}$ is the right Burt-Butler algebra of the bocs $(A,U)$ associated to $\poset$.
\end{theorem}
\begin{proof}
Consider the embedding of $K$-algebras $\mathcal{R} \to \End_K(e_0A)^{op}$ obtained from Lemma~\ref{L:Rfaith}, and the usual identification $\End_K(e_0A)^{op} \to \Mat_{n+1}(K)$ given through the ordered basis $(\alpha_1,\ldots,\alpha_n,e_0)$. Denote by $G(f)$ the matrix corresponding to $f \in \mathcal{R}$, and note that $G(f)_{ji}=f^{x_j,x_i}$ through this identification, where $x_i:=\alpha_i$ if $i\leq n$ and $x_{n+1}:=e_0$. Then the claims in Lemma~\ref{L:Rfaith} are translated to having a monomorphism of $K$-algebras,     
\[
G: \mathcal{R} \longrightarrow \begin{pmatrix} K\poset & 0 \\ \Mat_{1 \times n}(K)&K \end{pmatrix}.
\]
To count dimensions, using~(\ref{Eq:AV}) we have ${}_AU=A\omega \oplus \bigoplus_{i \prec j}Ae_i \otimes_K e_jA \cong A \oplus  \bigoplus_{i \prec j}Ae_i$ as left $A$-modules (since $e_jA$ is one-dimensional). Therefore, $\mathcal{R}=\Hom_A({}_AU,{}_AA)\cong A \oplus \bigoplus_{i \prec j}e_iA$, and taking dimensions over $K$ we get $\dimk_K\mathcal{R}=1+2n+\Card(i\prec j, \poset)$. Moreover, 
\begin{eqnarray*}
\dimk_K \begin{pmatrix} K\poset & 0 \\ \Mat_{1 \times n}(K)&K \end{pmatrix}&=&\dimk_K K\poset +\dimk_K \Mat_{1 \times n}(K) + \dimk_K K \\
&=&n+\Card(i\prec j,\poset)+n+1=\dimk_K \mathcal{R}.
\end{eqnarray*}
Then $G$ is an isomorphism. The structural constants of $\TO{z}$ for $z \in \TO{A}$ can be easily computed: $x_i \cdot \TO{e_j}=x_ie_j=\delta_{i,j}x_i$, and $x_i \cdot \TO{\alpha_j}=x_i\alpha_j=\delta_{i,n+1}x_j$. That is, $G(\TO{e_i})=E_{ii}$ for $i=1,\ldots,n+1$, and $G(\TO{\alpha_j})=E_{n+1,j}$ for $j=1,\ldots,n$, which shows that $A \cong \TO{A} \cong G(\TO{A})=\left( \begin{smallmatrix} R' & 0 \\ \Mat_{1 \times n}(K)&K \end{smallmatrix} \right)$, as claimed.
\end{proof}

%--------------------------------------
%--------------------------------------
\subsection{The left algebra} \label{B2}

In general, the left $\omega$-action of $\mathcal{L}$ on $A$ is not faithful. To follow the approach of the right case we consider a new grouplike element $\omega':=\omega+v'$ with $v':=\sum_{i \prec j}v_{ij}$. Note that $\epsilon(\omega')=\epsilon(\omega)=1$ and
\[
\delta(v')=\sum_{i \prec k}\delta(v_{ik})=\sum_{i \prec j \prec k}v_{ij} \otimes v_{jk}=v' \otimes v',
\] 
and therefore, $\mu(\omega')=\mu(\omega+v')=\omega\otimes\omega + \omega \otimes v'+v' \otimes\omega+v'\otimes v'=\omega' \otimes \omega'$, as claimed. Observe also that
\[
e_0 \omega' =\omega'e_0, \qquad e'\omega'=\omega'e' \qquad \text{and} \qquad \alpha'\omega'=\omega'\alpha', \qquad \text{where $\alpha':=\alpha_1+\ldots+\alpha_n$.}
\]
As before, for a left $\mathcal{L}$-module $M$ with (finite) $K$-basis $\mathbb{B}$ whose elements belong to $e_iM$ for some $i$, the action of $f \in \mathcal{L}$ on $M$ is determined by the structural constants $\{f^{b',b}\}_{b,b' \in \mathbb{B}}$, given by $f \cdot b =\sum_{b' \in \mathbb{B}}f^{b',b}b'$ for $b \in \mathbb{B}$. Moreover, if $f=e_jfe_i$ for $i,j \in \{0,1,\ldots,n\}$, then $f^{b,b'}=f^{e_jb',e_ib}$ for $b, b' \in \mathbb{B}$.

\begin{lemma}\label{L:Lfaith}
The left $\omega'$-action of $\mathcal{L}$ on $Ae'$ is faithful and satisfies $\mathcal{L}\cdot \alpha'=K\alpha'$. Moreover, the structural constants of this action with respect to the $K$-basis $\mathbb{B}'=\{e_1,\ldots,e_n,\alpha_1,\ldots,\alpha_n\}$ of $Ae'$ satisfy that $f^{x_j,y_i} \neq 0$ implies $i \preceq j$ for any $f \in \mathcal{R}$ and any $x_j \in \{e_j,\alpha_j\}$ and $y_i \in \{e_i,\alpha_i\}$.
\end{lemma}
\begin{proof}
Using~(\ref{Eq:differential}), the left $\omega'$-actions  of $f$ on $e_i$ and $\alpha_i$ are given by
\[
f \cdot e_i = f(e_i \omega')=f(\omega e_i +\delta(e_i)+e_iv')=f(\omega)e_i+\sum_{i\prec j}f(v_{ij})e_j
=\sum_{i\preceq j}(f^{e_j,e_i}e_j+f^{\alpha_j,e_i}\alpha_j),
\]
and
\[
f \cdot \alpha_i = f(\alpha_i \omega')=f(\omega\alpha_i +\delta(\alpha_i)+\alpha_iv')=f(\omega)\alpha_i-\sum_{h\prec i}f(\alpha_hv_{hi})e_i+\sum_{i\prec j}f(\alpha_iv_{ij})e_j=\sum_{i\preceq j}(f^{e_j,\alpha_i}e_j+f^{\alpha_j,\alpha_i}\alpha_j).
\]
This shows the claim on structural constants. Since $\alpha'\omega'=\omega'\alpha'$, then $f \cdot \alpha'=f(\alpha' \omega')=f(\omega')\alpha'$ for any $f \in \mathcal{L}$. Observe that $A\alpha'=K\alpha'$, which shows that $\mathcal{L}\cdot \alpha'=K\alpha'$. 

Assume now that $f \cdot z=0$ for all $z \in Ae'$. By the description of $f \cdot e_i$ we have $f(e_i \omega )=0=f(v_{ij})$ for all $i$ in $\poset$ and $i \prec j$. Similarly, from the description of $f \cdot \alpha_i$ we get $f(\alpha_iv_{ij})=0$ for all $i$ in $\poset$ and $i \prec j$. Finally, we have $0=f\cdot \alpha'=f(\alpha' \omega')=f(e_0\omega'\alpha')=f(e_0\omega +v'\alpha')=f(e_0 \omega)$. Then $f=0$, which completes the proof.
\end{proof}

Recall that $K^{\ones}\poset$ denotes the row-balanced incidence algebra of $\poset$, and $K^{\ones}_0\poset$ its row-centralized subspace, see subsection~\ref{N3}.

\begin{theorem}\label{T:LBB}
For any finite poset $\poset$ there are isomorphisms of $K$-algebras
\[
A \cong\begin{pmatrix} R' & 0 \\ R'& K\Id \end{pmatrix}  \longhookrightarrow  \begin{pmatrix} K\poset & K^{\ones}_0\poset \\ K\poset&K^{\ones}\poset \end{pmatrix} \cong \mathcal{L}^{\poset},
\]
where $\mathcal{L}^{\poset}$ is the left Burt-Butler algebra of the bocs $(A,U)$ associated to $\poset$.
\end{theorem}
\begin{proof}
Considering the $K$-basis $\mathbb{B}'=\{e_1,\ldots,e_n,\alpha_1,\ldots,\alpha_n\}$ of $Ae'$, the left $\omega'$-action of $\mathcal{L}$ on $Ae'$ produces a morphism of $K$-algebras $G:\mathcal{L} \to \Mat_{2n}(K)$, which is injective since this action is faithful by Lemma~\ref{L:Lfaith}. The entries of $G(f)$ are given by the structural constants of $f$ as follows,
\[
G(f)_{s,t}= \left\{
\begin{array}{ll}
f^{e_j,e_i}, & \text{if $i:=s\leq n$ and $j:=t \leq n$},\\
f^{\alpha_j,e_i}, & \text{if $i:=s\leq n$ and $j:=t-n>0$},\\
f^{e_j,\alpha_i}, & \text{if $i:=s-n>0$ and $j:=t \leq n$},\\
f^{\alpha_j,\alpha_i}, & \text{if $i:=s-n>0$ and $j:=t-n>0$},
\end{array} \right. \qquad \text{for $s,t \in \{1,\ldots,2n\}$.}
\]
In particular, the restrictions on structural constants stated in Lemma~\ref{L:Lfaith} say that the image of $G$ is contained in the matrix algebra $\left( \begin{smallmatrix} K\poset&K\poset\\K\poset&K\poset \end{smallmatrix} \right)$. Moreover, since $\alpha'$ corresponds to the vector $[0,\ones^{\tr}]^{\tr} \in K^{2n}$ under the basis $\mathbb{B}'$, then Lemma~\ref{L:Lfaith} imposes the condition 
\[
G(f)\begin{pmatrix} 0 \\ \ones \end{pmatrix} \in K\begin{pmatrix} 0 \\ \ones \end{pmatrix} \qquad \text{for any $f \in \mathcal{L}$},
\]
which shows that $G$ is a monomorphism whose image is contained in the matrix algebra $\left( \begin{smallmatrix} K\poset&K^{\ones}_0\poset\\K\poset&K^{\ones}\poset \end{smallmatrix} \right)$. To count dimensions, using~(\ref{Eq:AV}) we have $U_A=\omega A \oplus \bigoplus_{i \prec j}Ae_i \otimes_K e_jA \cong A \oplus  \bigoplus_{i \prec j}(e_jA)^2$ as right $A$-modules (since $Ae_i$ is two-dimensional). Therefore, $\mathcal{L}=\Hom_A(U_A,A_A)\cong A \oplus \bigoplus_{i \prec j}(Ae_j)^2$, and taking dimensions over $K$ we get $\dimk_K\mathcal{L}=1+2n+4\Card(i\prec j, \poset)$. Moreover, 
\begin{eqnarray*}
\dimk_K \begin{pmatrix} K\poset & K^{\ones}_0\poset \\ K\poset&K^{\ones}\poset \end{pmatrix}&=&2\dimk_K K\poset+\dimk_K K_0^\ones\poset +\dimk_K K^\ones \poset \\
&=&2n+2\Card(i\prec j,\poset)+\Card(i\prec j,\poset)+\Card(i\prec j,\poset)+1=\dimk_K \mathcal{L}.
\end{eqnarray*}
Then $G$ is an isomorphism. The left action of $\OT{A}$ on $Ae'$ is given by $\OT{z} \cdot b=zb$ for $z \in A$ and $b \in \mathbb{B}'$. Since $e_ib \neq 0$ or $\alpha_ib \neq 0$ implies $b=e_i$ then $G(\OT{e_i})=E_{ii}$ and $G(\OT{\alpha_i})=E_{i+n,i}$ for $i \in \poset$. Since $\OT{e_0} \cdot e_i=0$ and $\OT{e_0} \cdot \alpha_i=\alpha_i$ for all $i \in \poset$, then $G(\OT{e_0})=\sum_{i \in \poset}E_{i+n,1+n}$. This shows that
\[
G(\OT{A})=\begin{pmatrix} R' & 0 \\ R'& K\Id \end{pmatrix},
\]
and hence our last claim follows since $A \cong \OT{A}$.
\end{proof}

%-------
\subsubsection{Simple and injective $\mathcal{L}$-modules}

Recall from subsection~\ref{N3} that $\{ \eps_j \mid j \in \poset' \cup \{0\} \}$ is a complete set of orthogonal primitive idempotents of $K^{\ones}\poset$, where $\poset':=\poset-\min(\poset)$.  The identity $\Id_{2n}$ of $\mathcal{L}=\left( \begin{smallmatrix} K\poset & K^\ones_0 \poset \\ K\poset & K^\ones \poset \end{smallmatrix} \right)$ decomposes as sum of the following set of orthogonal primitive idempotents,
\begin{equation}\label{Eq:idempotents}
X:=\left\{ I_i:=\begin{pmatrix} E_{ii}&0\\0&0 \end{pmatrix}, J_k:=\begin{pmatrix} 0&0\\0&\eps_k \end{pmatrix},J_0:=\begin{pmatrix} 0&0\\0&\eps_0 \end{pmatrix} \mid i \in \poset, k \in \poset' \right\}.
\end{equation}
Observe that $I_k$ and $J_k$ are isomorphic idempotents of $\mathcal{L}$ for $k \in \poset'$, since $E_{kk}\eps_k=\eps_k$, $\eps_kE_{kk}=E_{kk}$, and
\begin{equation}\label{Eq:idemEquiv}
\begin{pmatrix} 0&\eps_k\\0&0 \end{pmatrix}\begin{pmatrix} 0&0\\E_{kk}&0 \end{pmatrix}=\begin{pmatrix} E_{kk}&0\\0&0 \end{pmatrix}=I_k
\qquad \text{and} \qquad
\begin{pmatrix} 0&0\\E_{kk}&0 \end{pmatrix}\begin{pmatrix} 0&\eps_k\\0&0 \end{pmatrix}=\begin{pmatrix} 0&0\\0&\eps_k \end{pmatrix}=J_k,
\end{equation}
and these are all isomorphisms within the set $X$. In particular, the dimensions over $K$ of the simple $\mathcal{L}$-modules as indexed by the set $\poset \cup \{0\}$ are given by
\[
\dimk_K S_\mathcal{L}(i)= \left\{
\begin{array}{ll}
1, & \text{if $i \in \min(\poset)$},\\
2, & \text{if $i \in \poset':=\poset-\min(\poset)$},\\
1, & \text{if $i=0$}.
\end{array} \right.
\]
For instance, for $i \in \poset'$ we have
\begin{equation} \label{Eq:simple}
S_\mathcal{L}(i)=\Id_{2n}S_\mathcal{L}(i)=(I_i+J_i)S_\mathcal{L}(i)=I_iS_\mathcal{L}(i) \oplus J_iS_\mathcal{L}(i),
\end{equation}
and each such direct summand is simple over a basic algebra, hence they have common dimension $1$.

The indecomposable injective $\mathcal{L}$-modules with socle different than $S_\mathcal{L}(0)$ are given by
\begin{equation}\label{Eq:injectiveL}
Q_\mathcal{L}(i)=D(I_i\mathcal{L}) \cong D\begin{pmatrix} E_{ii}K\poset & E_{ii}K^\ones_0\poset \\ 0&0 \end{pmatrix}.
\end{equation}
In particular, if $i \in \min(\poset)$ then $E_{ii}K^\ones_0\poset=0$ and $E_{ii}K\poset =KE_{ii}$, therefore  $Q_\mathcal{L}(i) \cong D(E_{ii}K\poset) \cong S_\mathcal{L}(i)$.

%-------
\subsubsection{A basic version for $\mathcal{L}$}

Here we exhibit a basic $K$-algebra $\mathcal{L}'$ Morita equivalent to $\mathcal{L}$. Take $n':=|\poset'|$ and $n'':=|\min(\poset)|$, and consider the following block forms of the incidence algebra $K\poset$ and its row-centralized subspace $K^{\ones}_0\poset$,
\[
K\poset = 
\begin{pNiceArray}{w{c}{5mm} w{c}{5mm}}
\RowStyle[cell-space-limits=5pt]{}
\Block[borders={bottom,right,tikz=dashed}]{1-1}{}
R'' & 0 \\
\RowStyle[cell-space-limits=5pt]{}
Z_{\poset}&
\Block[borders={top,left,tikz=dashed}]{1-1}{}
K\poset'
\end{pNiceArray} \qquad \text{and} \qquad
K^{\ones}_0\poset = 
\begin{pNiceArray}{w{c}{5mm} w{c}{5mm}}
\RowStyle[cell-space-limits=5pt]{}
\Block[borders={bottom,right,tikz=dashed}]{1-1}{}
0 & 0 \\
\RowStyle[cell-space-limits=5pt]{}
Z&
\Block[borders={top,left,tikz=dashed}]{1-1}{}
Z'
\end{pNiceArray}
\]
where $R''$ is the set of $n'' \times n''$ diagonal matrices, and $Z_{\poset}:=\{ M \in \Mat_{n' \times n''}(K) \mid \text{$M_{ji} \neq 0$ implies $i \prec j+n''$}\}$ (the shape of $Z$ and $Z'$ is irrelevant for our construction).  Removing the idempotents $I_j$ for $j \in \poset'$, that is, multiplying on right and left by the difference $\wt{\eps}=\Id_{2n}-\sum_{j \in \poset'}I_j$ we get the algebra
\[
\wt{\eps} \begin{pmatrix} K\poset & K^{\ones}_0\poset \\ K\poset&K^{\ones}\poset \end{pmatrix}\wt{\eps} =
\wt{\eps}\begin{pNiceArray}{w{c}{5mm} w{c}{5mm} | w{c}{5mm} w{c}{5mm}}
\CodeBefore
  \rectanglecolor{blue!12}{2-1}{2-4}
  \rectanglecolor{blue!12}{1-2}{4-2}
\Body
\RowStyle[cell-space-limits=5pt]{}
\Block[borders={bottom,right,tikz=dashed}]{1-1}{}
R'' & 0 & \Block[borders={bottom,right,tikz=dashed}]{1-1}{}0 & 0 \\
\RowStyle[cell-space-limits=5pt]{}
Z_{\poset}&
\Block[borders={top,left,tikz=dashed}]{1-1}{}
K\poset' & Z & \Block[borders={top,left,tikz=dashed}]{1-1}{} Z' \\
\hline
\RowStyle[cell-space-limits=5pt]{}
\Block[borders={bottom,right,tikz=dashed}]{1-1}{}
R'' & 0 &  \Block{2-2}<\large>{K^{\ones}\poset} & \\
\RowStyle[cell-space-limits=5pt]{}
Z_{\poset}&
\Block[borders={top,left,tikz=dashed}]{1-1}{}
K\poset' & & 
\end{pNiceArray}\wt{\eps}=
\begin{pNiceArray}{w{c}{5mm} w{c}{5mm} | w{c}{5mm} w{c}{5mm}}
\CodeBefore
  \rectanglecolor{blue!12}{2-1}{2-4}
  \rectanglecolor{blue!12}{1-2}{4-2}
\Body
\RowStyle[cell-space-limits=5pt]{}
\Block[borders={bottom,right,tikz=dashed}]{1-1}{}
R'' & 0 & \Block[borders={bottom,right,tikz=dashed}]{1-1}{}0 & 0 \\
\RowStyle[cell-space-limits=5pt]{} 0 &
\Block[borders={top,left,tikz=dashed}]{1-1}{}
0 & 0 & \Block[borders={top,left,tikz=dashed}]{1-1}{} 0 \\
\hline
\RowStyle[cell-space-limits=5pt]{}
\Block[borders={bottom,right,tikz=dashed}]{1-1}{}
R'' & 0 &  \Block{2-2}<\large>{K^{\ones}\poset} & \\
\RowStyle[cell-space-limits=5pt]{}
Z_{\poset}&
\Block[borders={top,left,tikz=dashed}]{1-1}{} 0 & & 
\end{pNiceArray} \cong
\begin{pNiceArray}{w{c}{5mm} | w{c}{5mm} w{c}{5mm}}
\RowStyle[cell-space-limits=5pt]{}
R'' & \Block[borders={right,tikz=dashed}]{1-1}{}0 & 0 \\
\hline
\RowStyle[cell-space-limits=5pt]{}
\Block[borders={bottom,right,tikz=dashed}]{1-1}{}
R'' &  \Block{2-2}<\large>{K^{\ones}\poset} & \\
\RowStyle[cell-space-limits=5pt]{}
Z_{\poset}& & 
\end{pNiceArray}.
\]
Since each removed idempotent is isomorphic to a non-removed idempotent~(\ref{Eq:idemEquiv}), the obtained $K$-algebra $\mathcal{L}'$ is Morita equivalent to the left algebra $\mathcal{L}$.% see Theorem~\ref{T:LBB}. 
Moreover, $\mathcal{L}'$ is basic since it is contained in the algebra of $(n''+n) \times (n''+n)$ lower triangular matrices with coefficients in $K$. We have proved the following:

\begin{remark}\label{R:basic}
For a finite poset $\poset$, the following matrix $K$-algebra is basic and Morita equivalent to the left algebra $\mathcal{L}$ of the bocs associated to $\poset$,
\[
\mathcal{L}' \cong \begin{pNiceArray}{w{c}{5mm} | w{c}{5mm} w{c}{5mm}}
\RowStyle[cell-space-limits=5pt]{}
R'' & \Block[borders={right,tikz=dashed}]{1-1}{}0 & 0 \\
\hline
\RowStyle[cell-space-limits=5pt]{}
\Block[borders={bottom,right,tikz=dashed}]{1-1}{}
R'' &  \Block{2-2}<\large>{K^{\ones}\poset} & \\
\RowStyle[cell-space-limits=5pt]{}
Z_{\poset}& & 
\end{pNiceArray},
\]
where $R''$ is the set of diagonal $n'' \times n''$-matrices with $n'':=|\min(\poset)|$, and $Z_{\poset}:=\{M \in \Mat_{n'' \times n'} \mid \text{$M_{ji} \neq 0$ implies $i \preceq j+n''$ in $\poset$}\}$ with $n':=|\poset-\min(\poset)|$.
\end{remark}

%--------------------------------------
%--------------------------------------
%--------------------------------------
%--------------------------------------
\section{Quasi-hereditary structure of incidence algebras} \label{S(QH)}

Quasi-hereditary algebras were introduced by Cline-Parshall-Scott~\cite{CPS88} for the study of highest weight categories in the general context of algebraic groups and Lie theory, and have since then been widely investigated also in module theoretical terms, for instance in Dlab-Ringel's standardization theorem~\cite{DR90}. On the other hand, the classical Poincar\'e-Birkhoff-Witt theorem in the representation theory of semisimple complex Lie algebras led Koenig to the notion of exact Borel subalgebras of a quasi-hereditary algebra~\cite{sK95}, a definition that was linked to the representation theory of so-called \emph{directed} bocses in the work of Koenig-K\"ulshammer-Ovsienko~\cite{KKO14}, see also~\cite{BKK20}.

In this section we assume that $K$ is an algebraically closed field.  Given a finite poset $\poset$ we denote by $\poset^{op}$ the \emph{opposite} poset of $\poset$. We also consider the \emph{upper suspension} $\poset^*:=\poset \cup \{*\}$ with partial order extending that of $\poset$ by setting $i \prec *$ for all $i \in \poset$, and take similarly $\poset_0$ by adding $0$ to $\poset$ as minimum element. 

%--------------------------------------
%--------------------------------------
\subsection{Regular and directed bocses} \label{QHm}

We start with some technical conditions on bocses whose relevance in the theory of quasi-hereditary algebras was established in~\cite{KKO14}, see also~\cite{jK}. Let $\Lambda$ be a finite-dimensional $K$-algebra and fix a complete set of representatives $\{S_{\Lambda}(i)\}_{i \in \poset}$ of isomorphism classes of simple left $\Lambda$-modules indexed by a finite poset $\poset$. Fix also projective covers $P_{\Lambda}(i)$ and injective envelopes $Q_{\Lambda}(i)$ of $S_{\Lambda}(i)$ for $i \in \poset$ (when convenient, we drop the subindex $\Lambda$). Recall that $\Lambda$ is called \emph{directed} with respect to $\poset$ if the isomorphism classes of simple left $\Lambda$-modules is indexed by $\poset$ and $\Hom_{\Lambda}(P(j),P(i)) \neq 0$ implies $i \preceq j$.

A bocs $(A,U)$ is \emph{directed} with respect to a poset $\poset$ if the algebra $A$ is directed with respect to $\poset$, the counit $\epsilon:U \to A$ is surjective, and its kernel $\Ker(\epsilon)$ is direct sum of finitely many projective $A-A$-bimodules of the form 
\[
P_A(j) \otimes_K D(Q_A(i)), \qquad \text{with $i\prec j$ in $\poset$,}
\]
where $D$ denotes the standard duality.

A normal bocs $(A,U)$ with differential $\delta$ is called \emph{regular} if $ \delta(\rad(A)) \subseteq \rad_{A\otimes A^{op}}(V)$ where $\rad(A)$ is the Jacobson radical of $A$ and $V$ is the kernel of the count $\epsilon$. Note that if $(A,U)$ is non-regular, then we may find $v \neq v_k$ generators of $V$, $a \in \rad(A)$ and $a'_k,a''_k \in A$ with $\delta(a)=\lambda v+\sum_j a''_kv_ka'_k$ for some $\lambda \in K$.

\begin{lemma}\label{P:directed}
The bocs associated to a finite poset $\poset$ is regular and directed with respect to $(\poset_0)^{op}$.
\end{lemma}
\begin{proof}
Let $(A,U)$ be the bocs associated to $\poset$. Besides isomorphisms, the only morphism between indecomposable projective $A$-modules are given by (scalar multiples) of right multiplication by arrows in $A$ (since $\rad^2(A)=0$), therefore they have domain $P_A(0)$ and codomain $P_A(i)$ for some $i \in \poset$. In particular, the base algebra $A$ is directed with respect to $(\poset_0)^{op}$. 

By construction, $(A,U)$ is a normal bocs, therefore its counit $\epsilon$ is surjective. By~(\ref{Eq:AV}), the kernel $V$ of the comultiplication $\epsilon$ decomposes as direct sum of finitely many modules of the form $Ae_i \otimes_K e_jA$ with $j\prec i$ in $(\poset_0)^{op}$, as wanted. For the regularity of $(A,U)$ note simply that there are no parallel arrows in the quiver defining $A$ and $U$.
\end{proof}

%--------------------------------------
%--------------------------------------
\subsection{Quasi-hereditary algebras and their exact Borel subalgebras} \label{QH0}

Let $\Lambda$ be a finite-dimensional $K$-algebra whose isomorphism classes of simple left $\Lambda$-modules are indexed by a poset $\poset=(\{1,\ldots,n\},\preceq )$.  With respect to $\poset$, the \emph{standard module} $\Delta(i)$ is the largest factor module of $P(i)$ having only composition factors $S(j)$ with $j \preceq i$, and the \emph{co-standard module} $\nabla(i)$ is the largest submodule of $Q(i)$ having only composition factors $S(j)$ with $j \preceq i$. A finite dimensional left $\Lambda$-module $M$ is \emph{$\Delta$-filtered} if there is a sequence of submodules of $M$,
\[
0=M_0 \subset M_1 \subset \ldots \subset M_r-1 \subset M_r=M,
\]
such that for any $1 \leq i \leq r$ we have $M_i/M_{i-1}\cong \Delta(j)$ for some $j \in \poset$.  The full subcategory of $\Lambda\SlMimod$ determined by $\Delta$-filtered modules is denoted by $\mathcal{F}(\Delta_{\Lambda})$, or simply $\mathcal{F}(\Delta)$ if no confusion arises. Dually, one defines the full subcategory $\mathcal{F}(\nabla_{\Lambda})$ of $\Lambda\SlMimod$ determined by modules filtered by co-standard modules.

\begin{definition}\label{D:qh}
Let $\poset$ be a poset indexing the isomorphism classes of simple modules of a finite-dimensional $K$-algebra $\Lambda$. Then $\Lambda$ is called \emph{quasi-hereditary} with respect to $\poset$ if the following conditions are satisfied:
\begin{enumerate}[label={\textnormal{(\alph*)}},topsep=3px,parsep=0px]
 \item[a)] $\dim_K\End_{\Lambda}(\Delta(i))=1$ for each $i \in \poset$;
 \item[b)] the projective $\Lambda$-module $P(i)$ has a $\Delta$-filtration for each $i \in \poset$;
 \item[c)] if $\Ext^1_{\Lambda}(\Delta(i),\Delta(j))\neq 0$ then $i \prec j$ in $\poset$.
\end{enumerate} 
\end{definition}

For instance, if $A$ is a directed algebra with respect to $\poset$ then $A$ is quasi-hereditary with simple standard modules.  If $\Lambda$ is a quasi-hereditary $K$-algebra with respect to $\poset$, then the \emph{Ringel dual} of $\Lambda$ is the $K$-algebra $R(\Lambda,\poset):=\End_{\Lambda}(T)^{op}$ where $T=\bigoplus_{i \in \poset}T_i$ is the unique (up to isomorphism) multiplicity free $\Lambda$-module such that $\add(T)=\mathcal{F}(\Delta_{\Lambda}) \cap \mathcal{F}(\nabla_{\Lambda})$, the so-called \emph{characteristic module} of $\Lambda$. Then $R(\Lambda,\poset)$ is quasi-hereditary with respect to $\poset^{op}$, and its Ringel dual is isomorphic to $\Lambda$, cf.~\cite{xC99,KKO14}.

\begin{corollary} \label{C:Ringel}
Let $\poset$ be a finite poset. Then the upper suspension algebra $K\poset^*$ is quasi-hereditary with respect to $(\poset_0)^{op}$, and its Ringel dual is given by
\[
R(K\poset^*,(\poset_0)^{op}) \cong \mathcal{L'}=\begin{pNiceArray}{w{c}{5mm} | w{c}{5mm} w{c}{5mm}}
\RowStyle[cell-space-limits=5pt]{}
R'' & \Block[borders={right,tikz=dashed}]{1-1}{}0 & 0 \\
\hline
\RowStyle[cell-space-limits=5pt]{}
\Block[borders={bottom,right,tikz=dashed}]{1-1}{}
R'' &  \Block{2-2}<\large>{K^{\ones}\poset} & \\
\RowStyle[cell-space-limits=5pt]{}
Z_{\poset}& & 
\end{pNiceArray},
\]
see Remark~\ref{R:basic}.
\end{corollary}
\begin{proof}
By Lemma~\ref{P:directed}, the bocs $(A,U)$ associated to poset $\poset$ is directed with respect to $(\poset_0)^{op}$. By~\cite[Corollary~1.2]{KKO14}, the right and left Burt-Butler algebras $\mathcal{R}$ and $\mathcal{L}$ of $(A,U)$ are quasi-hereditary with respect to $(\poset_0)^{op}$ and $\poset_0$ respectively, and they are Ringel dual of each other up to Morita equivalence. Then the claim follows since $\mathcal{R}$ is isomorphic to $K\poset^*$ (Theorem~\ref{T:RBB}), $\mathcal{L}'$ is basic and Morita equivalent to $\mathcal{L}$ (Theorem~\ref{T:LBB} and Remark~\ref{R:basic}), and the Ringel dual of $K\poset^*$ is basic and Morita equivalent to $\mathcal{L}$.
\end{proof}

Assume that $A$ is a subalgebra of a finite dimensional $K$-algebra $\Lambda$ and consider the induction functor $\Lambda \otimes_A -: A\SlMimod \to \Lambda\SlMimod$ and the co-induction functor $\Hom_A(\Lambda, -): A\SlMimod \to \Lambda\SlMimod$. If $\Lambda$ is quasi-hereditary with respect to $\poset$, we say that $A$ is an \emph{exact} (resp. a \emph{co-exact})  \emph{Borel subalgebra} of $\Lambda$, if $A$ is directed with respect to $\poset$ and
\begin{enumerate}[label={\textnormal{(\alph*)}},topsep=3px,parsep=0px]
  \item the induction functor $\Lambda \otimes_A -$ is exact (resp. the co-induction functor $\Hom_A(\Lambda, -)$ is exact);
  \item $\Delta_{\Lambda}(i) \cong \Lambda \otimes_A S_A(i)$ (resp. $\nabla_{\Lambda}(i) \cong \Hom_A(\Lambda, S_A(i))$) for each $i \in \poset$.  
 \end{enumerate}
If $A$ is an exact (resp. a co-exact) Borel subalgebra of $\Lambda$ we say that $A$ is \emph{homological} if the morphisms
\[
 \Ext_A^k(M,N) \longrightarrow \Ext_{\Lambda}^k(F(M),F(N)),
\]
induced by the exact functor $F:=\Lambda \otimes_A -$ (resp. $F:=\Hom_A(\Lambda,-)$) are epimorphisms for $k \geq 1$, and isomorphisms for $k \geq 2$. We say that $A$ is \emph{regular} if the morphisms induced by the functor $F$,
\[
\Ext^k_A(S_A(i),S_A(j)) \longrightarrow \Ext^k_{\Lambda}(\Delta_{\Lambda}(i),\Delta_{\Lambda}(j)), \qquad \text{(resp. $\Ext^k_A(S_A(i),S_A(j)) \longrightarrow \Ext^k_{\Lambda}(\nabla_{\Lambda}(i),\nabla_{\Lambda}(j))$),}
\]
are isomorphisms for any $i,j \in \poset$ and $k \geq 1$.

\begin{corollary}\label{C:Borel}
Let $\poset$ be a finite poset and let $\mathcal{R}$ and $\mathcal{L}$ be the right and left Burt-Butler algebras of the bocs $(A,U)$ associated to $\poset$.  Then the following (isomorphic) subalgebras
\[
\mathcal{R} \cong \begin{pmatrix} K\poset & 0 \\ \Mat_{1 \times n}(K)&K \end{pmatrix}  \longhookleftarrow  \begin{pmatrix} R' & 0 \\ \Mat_{1 \times n}(K)&K \end{pmatrix}
 \cong A \cong 
 \begin{pmatrix} R' & 0 \\ R'& K\Id \end{pmatrix}  \longhookrightarrow  \begin{pmatrix} K\poset & K^{\ones}_0\poset \\ K\poset&K^{\ones}\poset \end{pmatrix} \cong \mathcal{L},
\]
are regular homological exact and co-exact Borel subalgebras of $\mathcal{R}$ and $\mathcal{L}$, respectively.
\end{corollary}
\begin{proof}
That the $K$-algebra $\TO{A}$ and $\OT{A}$ are exact and co-exact Borel subalgebras of $\mathcal{R}$ and $\mathcal{L}$ follows from~\cite[Corollary~11.4 and its proof]{KKO14}. Since the counit of $(A,U)$ is surjective, then $A \cong \TO{A} \cong \OT{A}$, and the indicated subalgebras correspond to $A$ under the isomorphisms given in Theorems~\ref{T:RBB} and~\ref{T:LBB}, respectively. Moreover, these subalgebras are homological, see~\cite[Theorem~10.5]{KKO14}, and their regularity follows from the regularity of the bocs $A,U$, see Lemma~\ref{P:directed}.
\end{proof} 

We use Corollary~\ref{C:Borel} and the properties of (co-)exact Borel subalgebras to describe (co-)standard modules. For the case of standard $\mathcal{R}$-modules, note that
\[
\mathcal{R} \otimes_A S_A(0) =\mathcal{R} \otimes_A Ae_0 \cong \mathcal{R}e_0 \cong \begin{pmatrix} 0&0\\0&K \end{pmatrix}=S_\mathcal{R}(0),
\] 
where the last isomorphism follows from Theorem~\ref{T:RBB}, and that for any $k \in \poset$ we have an exact sequence of $A$-modules,
\[
\xymatrix{0 \ar[r] & S_A(0) \ar[r] & Ae_k \ar[r] & S_A(k) \ar[r] & 0}.
\]
Then $\xymatrix{0 \ar[r] &\mathcal{R} \otimes_A S_A(0) \ar[r] & \mathcal{R} \otimes_A Ae_k \ar[r] & \mathcal{R} \otimes_A S_A(k) \ar[r] & 0}$ is an exact sequence isomorphic to
\[
\xymatrix{0 \ar[r] & S_\mathcal{R}(0) \ar[r] & \mathcal{R}e_k \ar[r] & \Delta(k) \ar[r] & 0}.
\]
This shows that the standard modules for $\mathcal{R}$ are $\Delta(0)=S_\mathcal{R}(0)$ and $\Delta(k)=\mathcal{R}e_k/S_\mathcal{R}(0)$ for $k \in \poset$.

For the description of the co-standard modules for $\mathcal{L}$, we identify the left algebra $\mathcal{L}$ with the matrix algebra $\begin{pmatrix} K\poset&K^\ones_0\poset\\ K\poset & K^\ones \poset \end{pmatrix}$ as in Theorem~\ref{T:LBB}, and take for $k \in \poset$ and $i \in \poset'$,
\[
I_k:=\begin{pmatrix} E_{kk}&0 \\ 0&0 \end{pmatrix}, \quad 
I_E:=\begin{pmatrix} \Id_n&0 \\ 0&0 \end{pmatrix}, \quad
I_0:=\begin{pmatrix} 0&0 \\ 0&\Id_n \end{pmatrix}, \quad
J_i:=\begin{pmatrix} 0&0 \\ 0&\eps_i \end{pmatrix} \quad \text{and} \quad
J_0:=\begin{pmatrix} 0&0 \\ 0&\eps_0 \end{pmatrix},
\]
see~(\ref{Eq:idempotents}). Observe that for $k \in \poset$, since $S_A(k)$ is an injective $A$-module then $\nabla(k)=\Hom_A(\mathcal{L},S_A(k))$ is an injective $\mathcal{L}$-module. Indeed, since the inclusion $A \hookrightarrow \mathcal{L}$ is homologically coexact (Corollary~\ref{C:Borel}), then for any $A$-module $M$ we have an epimorphism,
\[
\xymatrix{0=\Ext^1_A(M,S_A(k)) \ar[rr] & & \Ext_\mathcal{L}^1(F(M),F(S_A(k)))},
\]
where $F=\Hom_A(\mathcal{L},-)$, which shows that $\Ext_\mathcal{L}^1(F(M),F(S_A(k)))=0$. Then $\nabla(k)=F(S_A(k))$ is $\Ext$-injective in $\mathcal{F}(\nabla)$, therefore it is an injective $\mathcal{L}$-module. In particular, since $\nabla(k)$ is indecomposable then its socle is simple. Observe also that $KI_k=S_A(k)$ under the isomorphism of Theorem~\ref{T:LBB}, and there is an $A$-module decomposition $\mathcal{L}=KI_k \oplus N$ with $\rad(\mathcal{L})\subseteq N$. Then there is a morphism of $A$-modules $\varphi:\mathcal{L} \to S_A(k)$ with $\varphi(N)=0$, and for any $x \in \mathcal{L}$ we have
\[
\rad(\mathcal{L})\varphi(x)=\varphi(x\rad(\mathcal{L})) \subseteq \varphi(N)=0.
\]
This shows that $\varphi$ is a non-zero element of $I_k \soc(\nabla(k))$, and therefore, $\nabla(k) \cong Q_\mathcal{L}(k)$. Moreover, note that
\begin{equation}\label{Eq:Lquo}
\rad(A)\mathcal{L} =\begin{pmatrix} 0&0 \\ \Id_n &0 \end{pmatrix} \mathcal{L} =\begin{pmatrix} 0&0 \\ K\poset & K^\ones_0\poset \end{pmatrix} \qquad \text{and} \qquad
\mathcal{L}/\rad(A)\mathcal{L} \cong \begin{pmatrix} K\poset & K^\ones_0\poset \\ 0& K\Id_n \end{pmatrix}.
\end{equation}
Then
\begin{eqnarray*}
\nabla(0) & = & \Hom_A(\mathcal{L},S_A(0)) = \Hom_A(\mathcal{L}/\rad(A)\mathcal{L},S_A(0)) \\
& = & \Hom_A(e_0 \cdot \mathcal{L}/\rad(A)\mathcal{L},e_0 \cdot S_A(0)) = \Hom_A(I_0\mathcal{L}/\rad(A)\mathcal{L},S_A(0)).
\end{eqnarray*}
On the other hand,
\[
I_0\mathcal{L}/\rad(A)\mathcal{L} \cong I_0\begin{pmatrix} K\poset & K^\ones_0\poset \\ 0& K\Id_n \end{pmatrix}=\begin{pmatrix} 0&0 \\ 0& K\Id_n \end{pmatrix}.
\]
Therefore, $\dimk_K \nabla(0)=1$, and since $J_0\nabla(0) \neq 0$ we have $\nabla(0) \cong S_\mathcal{L}(0)$. For convenience, we collect in Table~\ref{Tb:QHstructure} the quasi-hereditary structure of the right algebra $\mathcal{R}$ and its Ringel dual the left algebra $\mathcal{L}$.

\begin{table}[!hbt]
\begin{center}
$\begin{NiceArray}{w{r}{2cm} w{c}{2mm} w{l}{1cm}  w{c}{4cm}  w{r}{1cm} w{c}{2mm} w{l}{2cm}}
\Block{1-3}{\text{Right algebra $\mathcal{R}$}} & & & & \Block{1-3}{\text{Left algebra $\mathcal{L}$}} & &  \\[2ex]
\Delta(k) & = & P(k)/S(0) & \text{} & Q(k) & = & \nabla(k) \\
\Delta(0) & = & S(0) & \text{} & S(0) & = & \nabla(0) \\
\Hom_\mathcal{R}(\Delta(j),\Delta(i)) & \cong & K & \text{\scriptsize if $i \preceq j$ in $\poset$ or $i=0=j$} & K & \cong & \Hom_\mathcal{L}(\nabla(j),\nabla(i)) \\
\Ext^1_\mathcal{R}(\Delta(j),\Delta(i)) & \cong & K & \text{\scriptsize if $i=0\neq j$} & K & \cong & \Ext^1_\mathcal{L}(\nabla(j),\nabla(i)) \\
\nabla(k) & = & S(k) & \text{} & E_k & = & \Delta(k) \\
\nabla(0) & = & Q(0) & \text{} & S(0) & = & \Delta(0) 
\end{NiceArray}$
\end{center}
\caption{Description of the standard and co-standard modules of the right and left algebras associated to a finite poset $\poset$ with respect to $(\poset_0)^{op}$ and $\poset_0$, respectively. Here we take $k \in \poset$ and $i,j \in \poset \cup \{0\}$. For the description of the standard module $\Delta(k)$ of the right algebra $\mathcal{R}$, note that $S(0)$ is isomorphic to the (simple) socle of $P(k)$ and thus the quotient $P(k)/S(0)$ is well-defined. For the description of the standard module $\Delta(k)$ of the left algebra $\mathcal{L}$, we take $E_k:=S(k)$ if $k \notin \min(\poset)$, and $E_k$ the (unique up to isomorphism) indecomposable module with composition factors $S(0)$ and $S(k)$ if $k \in \min(\poset)$.} \label{Tb:QHstructure}
\end{table}

%--------------------------------------
%--------------------------------------
\subsection{Induced and co-induced modules} \label{QH2}

The full subcategory of $\mathcal{R}\SlMimod$ determined by modules isomorphic to $\mathcal{R} \otimes_A N$ for some $A$-module $N$, the so-called \emph{induced} modules, is denoted by $\Ind(A,\mathcal{R})$. Similarly, the full subcategory of $\mathcal{L}\SlMimod$ determined by modules isomorphic to $\Hom_A(\mathcal{L},N)$ for some $A$-module $N$, \emph{co-induced} modules, is denoted by $\Cnd(A,\mathcal{L})$. The following equivalences of categories are the main connection between representations of bocses and highest weight categories as studied in~\cite{KKO14},
\begin{equation} \label{Eq:equivalence}
\mathcal{F}(\Delta_{\mathcal{R}}) = \Ind(A,\mathcal{R}) \cong \mathcal{B}\SlMimod \cong \Cnd(A,\mathcal{L}) = \mathcal{F}(\nabla_{\mathcal{L}}).
\end{equation}
The categories $\mathcal{F}(\Delta_{\mathcal{R}})$ and $\mathcal{F}(\nabla_{\mathcal{L}})$ have exact structures and almost split sequences, and the equivalences shown in~(\ref{Eq:equivalence}) are exact. In particular, these categories are closed under direct summands and extensions. Observe that in the case of the right algebra $\mathcal{R}$ of a poset $\poset$ we have $\rad\Delta(k) \in \mathcal{F}(\Delta_\mathcal{R})$ for every $k \in \poset \cup \{0\}$, which is Conde's criterion for the right algebra of a directed bocs to be basic \cite[Theorem~5.2]{tC21a}.

%--------------------
\subsubsection{Restrictions to the incidence algebra $K\poset$}

To determine directly the subcategories of induced and co-induced modules in our case, first we note that the incidence algebra $K\poset$ is contained in both the right and left Burt-Butler algebras of the bocs associated to a finite poset $\poset$:
\begin{equation}\label{Eq:restrict}
(\unit-e_0)\mathcal{R}(\unit-e_0) \cong K \poset \cong
(\unit-e_0)\mathcal{L}(\unit-e_0),
\end{equation}
cf. Theorems~\ref{T:RBB} and~\ref{T:LBB} and the identifications of $e_0$ with $\TO{e_0}$ and $\OT{e_0}$ in subsection~\ref{B0}. 

Recall that given a finite dimensional $K$-algebra $\Lambda$ with idempotent $\bas$ there is an additive and exact functor $\res_e:\Lambda\SlMimod \to \bas \Lambda \bas \SlMimod$ called \emph{idempotent restriction}, given on objects by $M \mapsto \bas M$ (see for instance~\cite[Theorem~I.6.8]{ASS06}). The idempotent restrictions corresponding to equations~(\ref{Eq:restrict}) yield additive exact functors
\[
\res_0^{\mathcal{R}}:\mathcal{R}\SlMimod \longrightarrow K\poset\SlMimod  \qquad \text{and} \qquad
\res_0^{\mathcal{L}}:\mathcal{L}\SlMimod \longrightarrow K\poset\SlMimod.
\]
For a left $\mathcal{R}$-module $M$ we take $\Tr_0(M):=\sum \Img(f)$ where the sum is taken over all morphisms $f:S(0) \to M$, and we define similarly $\Tr_0(M)$ if $M$ is a left $\mathcal{L}$-module.

\begin{remark}\label{R:restrict}
Let $\mathcal{L}$ be the left algebra of a finite poset $\poset$.
\begin{enumerate}[label={\textnormal{(\alph*)}},topsep=3px,parsep=0px]
 \item For any $k \in \poset$ we have $\res_0^\mathcal{L}(S_\mathcal{L}(k)) \cong S_{K\poset}(k)$ and $\res_0^\mathcal{L}(Q_\mathcal{L}(k)) \cong Q_{K\poset}(k)$.
 \item If $T,T'$ are semisimple $\mathcal{L}$-modules with $\Tr_0(T)=\Tr_0(T')=0$ then $T \cong T'$ if and only if $\res_0^\mathcal{L}(T) \cong \res^\mathcal{L}_0(T')$
 \item For any left $\mathcal{L}$-module $M$ we have $\soc (\res^\mathcal{L}_0(M)) = \res^\mathcal{L}_0(\soc(M))$.
\end{enumerate}
\end{remark}
\begin{proof}
With the matrix presentation of $\mathcal{L}$ we have $\res_0^\mathcal{L}(M)=I_EM$ for any $\mathcal{L}$-module $M$. Then
\[
\res^\mathcal{L}_0(S_\mathcal{L}(k))=I_ES_\mathcal{L}(k)=\left(\sum_{i \in \poset}I_i \right) S_\mathcal{L}(k)=I_kS_\mathcal{L}(k)=S_{K\poset}(k),
\]
cf.~(\ref{Eq:simple}). Similarly, using~(\ref{Eq:injectiveL}) we get
\[
\res^\mathcal{L}_0(Q_\mathcal{L}(k))=I_EQ_\mathcal{L}(k) \cong D(I_k\mathcal{L}I_E) =
D\begin{pmatrix} E_{kk}K\poset&0\\0&0 \end{pmatrix} \cong D(E_{kk}K\poset)=Q_{K\poset}(k).
\]
This shows (a). For (b), since $\Tr_0(T)=\Tr_0(T')=0$ there are $n_k,n'_k \geq 0$ for $k \in \poset$ such that
\[
T=\bigoplus_{k  \in \poset} n_kS_\mathcal{L}(k) \qquad \text{and} \qquad 
T'=\bigoplus_{k  \in \poset} n'_kS_\mathcal{L}(k).
\]
Using (a) we get $\res_0^\mathcal{L}(T)=\bigoplus_{k  \in \poset} n_kS_{K\poset}(k)$ and 
$\res_0^\mathcal{L}(T')=\bigoplus_{k  \in \poset} n'_kS_{K\poset}(k)$. Then $\res_0^\mathcal{L}(T) \cong \res_0^\mathcal{L}(T')$ if and only if $n_k=n'_k$ for $k \in \poset$, which holds if and only if $T \cong T'$. 

To show (c) observe first that $\soc(M)=\Tr_0(M) \oplus S$, where $S \cong \bigoplus_{k \in \poset}n_kS_\mathcal{L}(k)$ for some $n_k \geq 0$. Note in particular that $\soc(\res^\mathcal{L}_0(M)) \cong \bigoplus_{k \in \poset}n_kS_{K\poset}(k)$. On the other hand, using (a) we get
\[
\res^\mathcal{L}_0(\soc(M))=I_E\soc(M)=I_E\Tr_0(M)\oplus I_ES \cong \bigoplus_{k \in \poset}n_kS_{K\poset}(k).
\]
This shows that $\soc (\res^\mathcal{L}_0(M)) = \res^\mathcal{L}_0(\soc(M))$.
\end{proof}

%--------------------
\subsubsection{Characterization of (co-) induced modules}

We need some preliminary observations.

\begin{lemma}\label{L:A}
Let $\mathcal{L}$ be the left algebra of a finite poset $\poset$. For any $k \in \poset$ we have
\[
I_0 Q_\mathcal{L}(k) \subseteq \mathcal{L} I_EQ_\mathcal{L}(k)=\mathcal{L} \res_0^\mathcal{L}(Q_\mathcal{L}(k)).
\]
\end{lemma}
\begin{proof}
As in~(\ref{Eq:injectiveL}), we use the following identity 
\[
I_0Q(k)=I_0D(I_k\mathcal{L}) \cong D\begin{pmatrix} 0 & E_{ii}K^\ones_0\poset \\ 0&0 \end{pmatrix}.
\]
Let $\{x_j \}_{j \in J}$ be a $K$-basis of $K\poset^\ones_0$ and $\{x_j^* \}_{j \in J}$ the corresponding dual basis on $D(K^\ones_0\poset)$. Let $W$ be a $K$-subspace of $K\poset$ which is complement of $K^\ones_0\poset$, and for each $j \in J$ consider the $K$-linear transformation $\hat{x}_j:K\poset \to K$ that vanishes on $W$ and coincides with $x_j$ on $K^\ones_0$.

Take now $S_j,\wh{S}_j:I_k\mathcal{L} \to K$ given by $S_j\left(\begin{smallmatrix}a&b\\0&0\end{smallmatrix}\right)=\hat{x}_j(a)$ and $\wh{S}_j\left(\begin{smallmatrix}a&b\\0&0\end{smallmatrix}\right)=x^*(b)$. Clearly, $S_j \in I_EQ(k)$ and $\wh{S}_j \in I_0Q(k)$. Moreover, we have
\[
\begin{pmatrix}0&0\\ \Id_n&0\end{pmatrix}\left[ S_j\begin{pmatrix}a&b\\0&0\end{pmatrix} \right]=
S_j\left[\begin{pmatrix}a&b\\0&0\end{pmatrix}\begin{pmatrix}0&0\\ \Id_n &0\end{pmatrix} \right]=S_j\begin{pmatrix}b&0\\0&0\end{pmatrix}=x^*_j(b)=\wh{S}_j\begin{pmatrix}a&b\\0&0\end{pmatrix}.
\]
Therefore, $\wh{S}_j=\left(\begin{smallmatrix}0&0\\ \Id &0\end{smallmatrix}\right)S_j \in \mathcal{L}I_EQ(k)$. This shows that $I_0Q(k) \subseteq \mathcal{L}I_EQ(k)$ since the elements $\wh{S}_j$ form a basis for $I_0Q(k)$.
\end{proof}

\begin{lemma}\label{L:trace}
Let $\poset$ be a finite poset with right and left algebras $\mathcal{R}$ and $\mathcal{L}$.
\begin{enumerate}[label={\textnormal{(\alph*)}},topsep=3px,parsep=0px]
 \item If $M$ is an indecomposable left $\mathcal{R}$-module with $\Tr_0(M)=0$ and such that the restriction $\res^{\mathcal{R}}_0(M)$ is a projective $K\poset$-module, then $M$ is a standard $\mathcal{R}$-module.
 \item If $M$ is an indecomposable left $\mathcal{L}$-module with $\Tr_0(M)=0$ and such that the restriction $\res^{\mathcal{L}}_0(M)$ is an injective $K\poset$-module, then $M$ is an injective (and therefore, a co-standard) $\mathcal{L}$-module.
\end{enumerate}
\end{lemma}
\begin{proof}
Claim (a) is straightforward since $S(0)=P(0)$ is a projective $\mathcal{R}$-module. Then $\Tr_0(M)=0$ means that $\bas'M=M$ and therefore, $\res_0^{\mathcal{R}}(M)$ is an indecomposable projective $K\poset$-module, say $P_{K\poset}(k)$. Then we have $M \cong P_{\mathcal{R}}(k)/\soc(P_{\mathcal{R}}(k))=\Delta(k)$, which is a standard module, see Table~\ref{Tb:QHstructure}.

To show (b), since $\Tr_0(M)=0$, by Remark~\ref{R:restrict}(a) there is an injective $\mathcal{L}$-module $Q$ with $\Tr_0(Q)=0$ and such that $\res^\mathcal{L}_0(M) \cong \res^\mathcal{L}_0(Q)$. Using Remark~\ref{R:restrict}(c) we get
\[
\res_0^\mathcal{L}(\soc (M))=\soc(\res_0^\mathcal{L}(M)) \cong \soc(\res_0^\mathcal{L}(Q))=\res_0^\mathcal{L}(\soc(Q)),
\]
and then $\soc(M) \cong \soc(Q)$ by Remark~\ref{R:restrict}(b). Since $Q$ is injective, there is a morphism $\iota:M \to Q$ such that the restriction $\res^\mathcal{L}_0(\iota)$ is an isomorphism. Note that $\iota$ is a monomorphism since so is $\res^\mathcal{L}_0(\iota)$ and $\Tr_0(M)=\Tr_0(Q)=0$. Since the idempotents $I_1,\ldots,I_n,J_0$ form a system of representatives up to isomorphism of orthogonal primitive idempotents~(\ref{Eq:idempotents}) with $I_E=\sum_{i=1}^nI_i$, to show that $\iota$ is an isomorphism it is enough to show that $J_0Q$ is contained in the image of $\iota$. Recall also that $I_0=J_0+\sum_{k \in \poset'}J_k$ is a decomposition of $I_0$ as sum of orthogonal idempotents, thus $J_0 M \subseteq I_0 M$ for any $\mathcal{L}$-module $M$. We have $Q = \bigoplus_{j \in \poset} n_kQ(k)$ for some $n_k \geq 0$, and by Lemma~\ref{L:A} we get
\begin{eqnarray*}
J_0Q & = & \bigoplus_{j \in \poset} n_kJ_0Q(k) \subseteq \bigoplus_{j \in \poset} n_k I_0Q(k) \subseteq \bigoplus_{j \in \poset} n_k\mathcal{L}I_EQ(k) \\
& = & \mathcal{L} \res_0^\mathcal{L}(Q) \subseteq \mathcal{L} \iota(M)=\iota(M).
\end{eqnarray*}
Then $\iota:M \to Q$ is an isomorphism, which completes the proof.
\end{proof}

\begin{theorem}\label{T:COinduced}
Let $\mathcal{R}$ and $\mathcal{L}$ be the right and left algebras of the bocs associated to a finite poset $\poset$. Then:
\begin{enumerate}[label={\textnormal{(\alph*)}},topsep=3px,parsep=0px]
 \item a left $\mathcal{R}$-module $M$ is induced if and only if its restriction $\res_0^{\mathcal{R}}(M)$ is a projective $K\poset$-module;
 \item a left $\mathcal{L}$-module $M$ is co-induced if and only if its restrictions $\res_0^{\mathcal{L}}(M)$ is an injective $K\poset$-module.
\end{enumerate}
\end{theorem}
\begin{proof}
To show (a), observe from Table~\ref{Tb:QHstructure} that $\res_0^{\mathcal{R}}(\Delta(k)) \cong P_{K\poset}(k)$ for any $k \in \poset$, and $\res_0^{\mathcal{R}}(\Delta(0))=0$. Since $\res_0^{\mathcal{R}}$ is additive and exact, then $\res_0^{\mathcal{R}}(M)$ is projective for any $M$ in $\mathcal{F}(\Delta_{\mathcal{R}})=\Ind(A,\mathcal{R})$. To show the converse, assume that $M$ is a $\mathcal{R}$-module with projective $\res_0^{\mathcal{R}}(M)$. By the exactness of the functor $\res_0^{\mathcal{R}}$, since $\res^{\mathcal{R}}_0(\Tr_0(M))=0$ we have exact sequences
\[
\xymatrix@C=1pc{0 \ar[r] & \Tr_0(M) \ar[r] & M \ar[r] & M/\Tr_0(M) \ar[r] & 0} \qquad \text{and} \qquad
\xymatrix@C=1pc{0 \ar[r] & \res_0^{\mathcal{R}}(M) \ar[r] & \res_0^{\mathcal{R}}(M/\Tr_0(M)) \ar[r] & 0}.
\]
In particular, $\res_0^{\mathcal{R}}(M/\Tr_0(M))$ is projective. Note that $\Tr_0(M/\Tr_0(M))=0$ since $\Ext^1_{\mathcal{R}}(S(0),S(0))=0$. Therefore, by additivity and Lemma~\ref{L:trace}(a) we have that $M/\Tr_0(M) \in \add(\Delta_{\mathcal{R}})$, and clearly $\Tr_0(M) \in \add(\Delta_{\mathcal{R}}(0))$. Then $M \in \mathcal{F}(\Delta_{\mathcal{R}})=\Ind(A,\mathcal{R})$, as wanted.

For the proof of (b) observe that $\res_0^{\mathcal{L}}(\nabla(k))$ is an injective $K\poset$-module  for any $k$ in $\poset$. This follows from $\nabla(k)$ being an injective $\mathcal{L}$-module, see Table~\ref{Tb:QHstructure} and~\cite[Theorem~I.6.8]{ASS06}. Then, as before, $\res_0^{\mathcal{L}}(M)$ is an injective $K\poset$-module for any $\mathcal{L}$-module $M$ in $\mathcal{F}(\nabla_{\mathcal{L}})=\Cnd(A,\mathcal{L})$. The converse implication is as in case (a) using Lemma~\ref{L:trace}(b).
\end{proof}

%--------------------------------------
%--------------------------------------
%--------------------------------------
%--------------------------------------
\section{Comments and examples} \label{S(X)}

In this section we give bound quiver presentations and examples of the right and left algebras of the bocs associated to a finite poset.

%--------------------------------------
%--------------------------------------
\subsection{Bound quiver for the right algebra} \label{X1}

A bound quiver presentation for the incidence algebra $K\poset$ is simple and well-known: consider the \emph{Hasse quiver} $\Hasse(\poset)$ of a finite poset ${\poset}$, which has as vertices the set ${\poset}$ and there is an arrow $\nu_{ji}:i \to j$ for any $i$ and any immediate successor $j$ of $i$. Consider the ideal $I_{\poset}$ of $K\Hasse(\poset)$ generated by the difference of all parallel paths in $\Hasse(\poset)$. This is an admissible ideal of $K\Hasse(\poset)$, and there is an isomorphism of $K$-algebras
\begin{equation}\label{Eq:incidence2}
K\poset \cong K\Hasse(\poset)/I_{\poset},
\end{equation}
determined by taking the matrix $E_{ii}$ to the trivial path at vertex $i$, and the matrix $E_{ji}$ with $j$ a direct successor of $i$ to the corresponding arrow $\nu_{ji}$ in $\Hasse(\poset)$.  With this presentation it is easy to give a bound quiver for the right algebra associated to a finite poset.

\begin{definition}\label{D:RBQ}
Let $\poset$ be a finite poset with right algebra $\mathcal{R}$.
\begin{enumerate}[label={\textnormal{(\alph*)}},topsep=3px,parsep=0px]
 \item The Gabriel quiver $\Quiver^R=(\Quiver^R_0,\Quiver^R_1)$:
\[
\begin{NiceArray}{w{c}{3mm} w{c}{2.5mm} w{l}{20mm} w{l}{30mm}}
\RowStyle[cell-space-limits=5pt]{}
\Quiver^R_0 & := & \poset \cup \{*\} & \\
\Block{2-1}{\Quiver^R_1} & \Block{2-1}{:=}
& \nu_{ji}: i \to j, & \text{for $i \to j$ in $\poset'$} \\
 & & a_j:j \to *, &  \text{for $j \in \max(\poset)$}
\CodeAfter
\SubMatrix{\{}{2-3}{3-4}{\}}
\end{NiceArray}
\]
 \item Relations $R^R:= \{\text{all commutative relations}\}$.
 \end{enumerate}
Then $\mathcal{R} \cong \Quiver^R/I^R$ with $I^R:=\langle R^R \rangle$.
\end{definition}

%--------------------------------------
%--------------------------------------
\subsection{Bound quiver for the left algebra} \label{X1P}

To describe the left algebra, first we produce a bound quiver presentation of the row-balanced incidence algebra $K^{\ones}\poset$. Recall that we have fixed a semi-simple subalgebra $R^{\ones}=\bigoplus_{j \in \poset' \cup \{0\}}K\eps_j$ of $K^{\ones}\poset$ and a basis of its radical
\[
\rad(K^{\ones}\poset)=\bigoplus_{\substack{ i \prec j \\ i \neq m(j)}}KB_{ji}.
\] 
 A direct calculation shows that the product in $K^{\ones}\poset$ on the chosen basis is given for $i\prec j$ by
\begin{enumerate}[label={\textnormal{(\alph*)}},topsep=3px,parsep=0px]
\item $B_{ji}=\eps_jB_{ji}\eps_i$ if $i\in \poset'$ and $B_{ji}=\eps_jB_{ji}\eps_0$ if $i \in \min(\poset)$.
\item If $i \in \poset'$ then $B_{kj}B_{ji}=B_{ki}$.
\item If $j\prec k$ then $B_{kj}(B_{ji_1}-B_{ji_2})=B_{ki_1}-B_{ki_2}$.
\end{enumerate}

We need some further notions to describe the square radical $\rad^2(K^{\ones}\poset)$. 
\begin{enumerate}[start=4,label={\textnormal{(\alph*)}},topsep=3px,parsep=0px]
\item If $h,h' \in \min(\poset,j)$ we write $h \sim_j h'$ if there is $i\prec j$ such that $h,h' \in \min(\poset,i)$. Clearly, $\sim_j$ is a symmetric relation on $\min(\poset,j)$.

\item Extend $\sim_j$ reflexively and transitively to obtain an equivalence relation $\simeq_j$ on $\min(\poset,j)$. To be precise, $h \simeq_j h'$ if and only if either $h=h'$ or there are $h_0,h_1,\ldots,h_r \in \min(\poset,j)$ such that
\[
h=h_0 \sim_j h_1 \sim_j h_2 \sim_j \ldots  \sim_j h_r=h'.
\] 

\item Denote by $\Orb(j)$ the number of equivalence classes in $\min({\poset},j)/\simeq_j$, and choose a set $\wt{\ell}(j)$ of representatives of those class not containing $m(j)$.
\end{enumerate}
Then the rules (a-c) yield the following form of the square radical of $K^{\ones}\poset$,
\[
\rad^2(K^{\ones}\poset)=\left[ \bigoplus_{\substack{i\prec j \\ i \not\to j }}KB_{ji} \right] \oplus \bigoplus_{j \in \poset'}\left[ \left( \bigoplus_{\substack{i \simeq_j m(j) \\ i \neq m(j) }}KB_{ji}  \right) \oplus \bigoplus_{i' \in \wt{\ell}(j)} \left( \bigoplus_{\substack{i \simeq_j i' \\ i \neq i' }}K(B_{ji}-B_{ji'})  \right) \right],
\]
(recall that $B_{jm(j)}=0$ for any $j \in \poset'$). With the description of $\rad(K^{\ones}\poset)$ and $\rad^2(K^{\ones}\poset)$ we give the Gabriel quiver of $K^{\ones}\poset$ as follows, see~\cite{ASS06}. Consider the set of matrices
\[
\mathbb{B}:=\{B_{ji} \mid \text{$i \to j$ in $\poset'$}\} \cup \{ B_{ji} \mid \text{$j \in \poset'$ and $i \in \wt{\ell}(j)$} \}.
\]
These are linearly independent matrices in the radical of $K^{\ones}\poset$ whose projections are a $K$-basis of the quotient $\rad(K^{\ones}\poset)/\rad^2(K^{\ones}\poset)$. By (a), the Gabriel quiver of $K^{\ones}\poset$ can be obtained from the Hasse quiver $\Hasse(\poset')$ of the poset $\poset':=\poset-\min(\poset)$ by adding a vertex $0$ and exactly $\Orb(j)-1 \geq 0$ arrows from $0$ to $j$ for any $j \in \poset$. To be precise, consider the quiver $\Hasse^{\ones}(\poset):=(\poset' \cup \{0\},\mathbb{B},\sou,\tar)$ where $\tar(B_{ji}):=j$ and
\[
\sou(B_{ji}) := \left\{
\begin{array}{ll}
i, & \text{if $i \in \poset'$},\\
0, & \text{if $i\in \wt{\ell}(j)$}.
\end{array} \right.
\]
By construction, there is a surjective morphism of $K$-algebras $F:K\Hasse^{\ones}(\poset) \to K^{\ones}\poset$. By (b), all commutative relations in $\Hasse(\poset')$ belong to the kernel of $F$. The set of such relations is denoted by $R_{\poset'}$. Equation~(c) produces the following additional relations: for any full subposet $\poset''$ of $\poset$ as below, with $i_1,i_2 \in \min(\poset)$, we consider the relation $B_{kj_1}(B_{j_1i_1}-B_{j_1i_2})-B_{kj_2}(B_{j_2i_1}-B_{j_2i_2})$,
\begin{equation}\label{rel1}
\xymatrix@C=1pc@R=1pc{& k \ar@{-}[ld] \ar@{-}[rd] \\ j_1 \ar@{-}[d] \ar@{-}[rrd] && j_2 \ar@{-}[d] \ar@{-}[lld] \\
i_1 && i_2 }
\end{equation}
The set of such relations is denoted by $R^{\ones}_{\min(\poset)}$. All together, we get an isomorphism of $K$-algebras
\begin{equation}\label{Eq:Rowincidence2}
K^{\ones}\poset \cong K\Hasse^{\ones}(\poset)/I^{\ones}_{\poset}, \qquad \text{with $I^{\ones}_{\poset}:=\langle R_{\poset'} \cup R^{\ones}_{\min(\poset)} \rangle$.}
\end{equation}
The precise shape of the relations in $R^{\ones}_{\min(\poset)}$ in terms of the arrows of $\Hasse^{\ones}(\poset)$ is highly dependent on the choice of minimal marking $m:\poset \to \min(\poset)$ and representatives $\wt{\ell}(j)$.

\begin{definition}\label{D:LBQ}
Let $\poset$ be a finite poset with left algebra $\mathcal{L}$.
\begin{enumerate}[label={\textnormal{(\alph*)}},topsep=3px,parsep=0px]
 \item The Gabriel quiver $\Quiver^L=(\Quiver^L_0,\Quiver^L_1)$:
\[
\begin{NiceArray}{w{c}{3mm} w{c}{2.5mm} w{l}{2cm} w{l}{6cm}}
\RowStyle[cell-space-limits=5pt]{}
\Quiver^L_0 & := & \poset \cup \{0\} & \\
\Block{3-1}{\Quiver^L_1} & \Block{3-1}{:=}
& \nu_{ji}: i \to j, & \text{for $i \to j$ in $\poset'$ or $\{i\}=\{h \in \poset \mid h \prec j\}$} \\
 & & a_i:i \to 0, &  \text{for $i \in \min(\poset)$} \\
 & & b_{ji}:0 \to j, &  \text{for $j \in \poset'$ and $i \in \wt{\ell}(j)$} 
\CodeAfter
\SubMatrix{\{}{2-3}{4-4}{\}}
\end{NiceArray}
\]
 \item For $i \in \min(\poset,j)$ and $j \prec k$ take
 \[
\nu_{k,j} := \left\{
\begin{array}{ll}
e_j, & \text{if $k=j$},\\
\nu_{kj}, & \text{if $j \to k$},\\
\nu_{kj_1}\cdots \nu_{j_rj}, & \text{if $j \to j_r \to \ldots \to j_1 \to k$},
\end{array} \right. \qquad
p(j,i) := \left\{
\begin{array}{ll}
\nu_{ji}, & \text{if $\{i\}=\{h \in \poset \mid h \prec j\}$},\\
-b_{ki'}a_i, & \text{if $i=m(j)$ and $i' \in \wt{\ell}(j)$},\\
b_{ki}a_i, & \text{if $i \in \wt{\ell}(j)$}, \\
\emptyset, & \text{otherwise}.
\end{array} \right.
 \]
We also consider the following paths for $i \prec k$ in $\poset$, $\Paths(k,i):=\{\nu_{k,j}p(j,i) \mid \text{$i \prec j \preceq k$ and $p(j,i)\neq \emptyset$} \}$.

 \item Relations:
\[
\begin{NiceArray}{w{c}{7mm} w{c}{2.5mm} w{l}{6cm} w{l}{5cm}}
\RowStyle[cell-space-limits=5pt]{}
R_{\poset'} & := & \{ \text{all commutative relations in $\Hasse(\poset')$}\}; \\
R^{\ones}_{\min(\poset)} & := & \{b_{kj_1}(b_{j_1i_1}-b_{j_1i_2})-b_{kj_2}(b_{j_2i_1}-b_{j_2i_2})\} & \text{as in~(\ref{rel1})}; \\
\Block{2-1}{R^0_{\min(\poset)}} & \Block{2-1}{:=}
& (b_{ji}-b_{ji'})a_{m(j)}, & \text{for $i,i' \in \wt{\ell}(j)$} \\
 & & b_{ji}a_{i'}, &  \text{for $i' \neq m(j)$ and $i' \neq i \in \wt{\ell}(j)$;} \\
\RowStyle[cell-space-limits=5pt]{} R_{\Paths} & := & \{\gamma-\gamma' \mid \text{$\gamma,\gamma' \in \Paths(k,i)$ and $i \in \min(\poset,k)$} \}; 
\CodeAfter
\SubMatrix{\{}{3-3}{4-4}{\}}
\end{NiceArray} 
\] 
 \end{enumerate}
Then the left algebra $\mathcal{L}$ is Morita equivalent to $\Quiver^L/I^L$ with $I^L:=\langle R_{\poset'} \cup R^{\ones}_{\min(\poset)} \cup R^{0}_{\min(\poset)} \cup R_{\Paths} \rangle$.
\end{definition}

%--------------------------------------
%--------------------------------------
\subsection{Some examples} \label{X2}

First we consider a family of examples, trivial for our discussion.

\begin{remark}
The following are equivalent conditions for a finite poset $\poset$:
\begin{enumerate}[label={\textnormal{(\alph*)}},topsep=3px,parsep=0px]
 \item $\poset=\min(\poset)$, that is, $\poset$ is an \emph{anti-chain};
 \item the bocs $(A,U)$ associated to $\poset$ has exactly one grouplike element;
 \item the right algebra satisfies $\mathcal{R} \cong A$;
 \item the left algebra satisfies $\mathcal{L} \cong A$;
 \item the left algebra $\mathcal{L}$ is basic.
\end{enumerate}
\end{remark}

We also consider some cases with ``simple minimal part'', where the basic left algebra $\mathcal{L}'$ is isomorphic to an incidence algebra. 

\begin{remark}
The simple $\mathcal{L}$-module $S(0)$ is projective if and only if every connected component of $\poset$ has a minimum. In this case the left algebra is Morita equivalent to the incidence algebra $K\poset_{\min,0}$ of the poset
\[
 \poset_{\min,0}:=(\poset \sqcup \{0\},\preceq')
\]
where $\preceq'$ extends $\preceq$ by taking $i \prec 0$ for any $i \in \min(\poset)$.
\end{remark}

Next we give examples of small posets, with right and basic left algebra shown below the corresponding Hasse diagram following Definitions~\ref{D:RBQ} and~\ref{D:LBQ}. In all cases we take the minimal marking $m(j)=1$ for $j \notin \min(\poset)$, with exception of vertex $4$ in the second poset, where we must take $m(4)=2$. Note that in these examples all left algebras are hereditary.
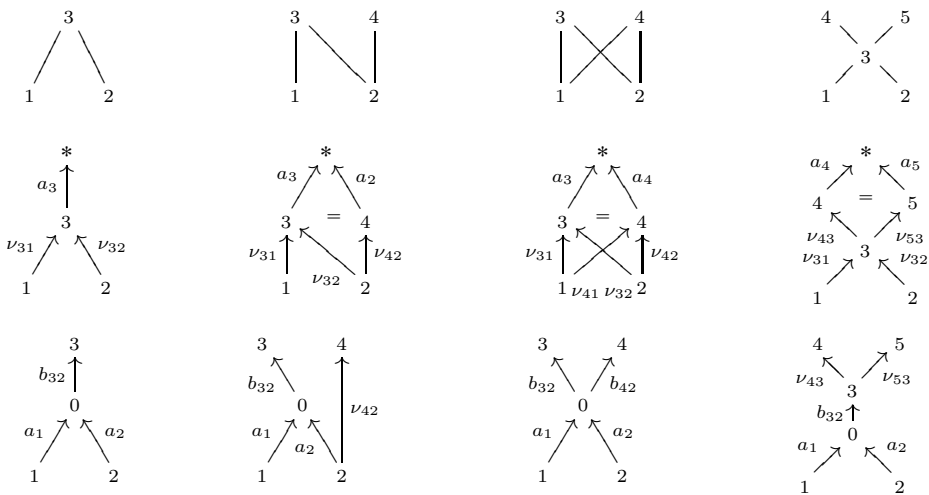
\[
\xymatrix@!0@C=15pt@R=15pt{ & {}_3 \ar@{-}[ldd] \ar@{-}[rdd] \\ \\ {}_1 & & {}_2 }\qquad\qquad \qquad
\xymatrix@!0@C=15pt@R=15pt{ {}_3 \ar@{-}[dd] \ar@{-}[ddrr]  & & {}_4 \ar@{-}[dd] \\ \\ {}_1 & & {}_2 } \qquad \qquad \qquad
\xymatrix@!0@C=15pt@R=15pt{ {}_3 \ar@{-}[dd] \ar@{-}[ddrr]  & & {}_4 \ar@{-}[dd] \ar@{-}[ddll] \\ \\ {}_1 & & {}_2 } \qquad \qquad \qquad
\xymatrix@!0@C=15pt@R=15pt{ {}_4 \ar@{-}[rd] & & {}_5 \ar@{-}[dl] \\ & {}_3 \ar@{-}[rd] \ar@{-}[dl]  \\ {}_1 & & {}_2 }
\]
\[
\xymatrix@!0@C=15pt@R=25pt{ & {*} \\ & {}_3 \ar[u]^-{a_3} \ar@{<-}[ld]_-{\nu_{31}} \ar@{<-}[rd]^-{\nu_{32}} \\ {}_1 & & {}_2 }\qquad\qquad
\xymatrix@!0@C=15pt@R=25pt{ & {*} \\ {}_3 \ar@{}[rr]|(.6){=} \ar[ur]^{a_3} \ar@{<-}[d]_-{\nu_{31}} \ar@{<-}[drr]_(.75){\nu_{32}}  & & {}_4 \ar[ul]_-{a_2} \ar@{<-}[d]^-{\nu_{42}} \\ {}_1 & & {}_2 } \qquad \qquad
\xymatrix@!0@C=15pt@R=25pt{ & {*} \\ {}_3 \ar@{}[rr]|-{=} \ar[ur]^-{a_3} \ar@{<-}[d]_-{\nu_{31}} \ar@{<-}[drr]_(.95){\nu_{32}}  & & {}_4  \ar[ul]_-{a_4} \ar@{<-}[d]^-{\nu_{42}} \ar@{<-}[dll]^(.95){\nu_{41}} \\ {}_1 & & {}_2 } \qquad \qquad
\xymatrix@!0@C=18pt@R=18pt{ & {*} \\ {}_4 \ar@{}[rr]|-{=} \ar[ur]^-{a_4} \ar@{<-}[rd]_-{\nu_{43}} & & {}_5 \ar[ul]_-{a_5} \ar@{<-}[dl]^-{\nu_{53}} \\ & {}_3 \ar@{<-}[rd]^-{\nu_{32}} \ar@{<-}[dl]_-{\nu_{31}}  \\ {}_1 & & {}_2 }
\]
\[
\xymatrix@!0@C=15pt@R=25pt{ & {{}_3} \\ & \mathmiddlescript{0} \ar[u]^-{b_{32}} \ar@{<-}[ld]_-{a_1} \ar@{<-}[rd]^-{a_2} \\ {}_1 & & {}_2 }\qquad\quad \quad
\xymatrix@!0@C=15pt@R=25pt{ {}_3 \ar@{<-}[dr]_-{b_{32}}  & & {}_4  \ar@{<-}[dd]^-{\nu_{42}} \\ & \mathmiddlescript{0} \\ {}_1 \ar[ru]^-{a_1} & & {}_2 \ar[lu]^-{a_2} } \qquad \qquad \quad
\xymatrix@!0@C=15pt@R=25pt{ {}_3 \ar@{<-}[rd]_-{b_{32}} & & {}_4 \ar@{<-}[dl]^-{b_{42}} \\ & \mathmiddlescript{0} \ar@{<-}[rd]^-{a_2} \ar@{<-}[dl]_-{a_1}  \\ {}_1 & & {}_2 } \qquad \qquad \quad
\xymatrix@!0@C=18pt@R=18pt{ \quad{}_4 \ar@{<-}[rd]_-{\nu_{43}} & & {}_5 \ar@{<-}[dl]^-{\nu_{53}} \\ & {}_3 \ar@{<-}[d]_-{b_{32}} \\ & \mathmiddlescript{0} \ar@{<-}[rd]^-{a_2} \ar@{<-}[dl]_-{a_1}  \\ {}_1 & & {}_2 }
\]

\bigskip

\noindent {\bf Statements and declarations} \\
The authors have no competing interests to declare that are relevant to the content of this article.

\bibliographystyle{plainnat}

\end{document}